\documentclass[a4paper,1p]{elsarticleNOFOOTNOTE}
\usepackage{amsmath,amssymb,anysize,float,url}
\usepackage{longtable}
\usepackage{framed}

\newtheorem{theorem}{Theorem}[section] 
\newtheorem{lemma}[theorem]{Lemma}     
\newtheorem{proposition}[theorem]{Proposition}
\newtheorem{problem}[theorem]{Problem}
\newtheorem{remark}[theorem]{Remark}
\newtheorem{example}[theorem]{Example}
\newproof{proof}{Proof}

\newcommand{\PG}{\mathsf{PG}}
\newcommand{\GF}{\mathsf{GF}}
\newcommand{\GL}{\mathsf{GL}}
\newcommand{\AGL}{\mathsf{AGL}}
\newcommand{\Aut}{\mathsf{Aut}}

\newcommand{\Sp}{\mathsf{Sp}}
\newcommand{\PGU}{\mathsf{PGU}}
\newcommand{\PGammaU}{\mathsf{P}\Gamma\mathsf{U}}

\newcommand{\SL}{\mathsf{SL}}
\newcommand{\GSp}{\mathsf{GSp}}

\newcommand{\gq}{generalized quadrangle}
\newcommand{\gqs}{generalized quadrangles}
\newcommand{\ftwkb}[1]{\mathsf{FTWKB}({#1})}
\newcommand{\kmonom}[1]{\mathsf{K}_2({#1})}
\newcommand{\kknuth}[1]{\mathsf{K}_1({#1})}
\newcommand{\fish}[1]{\mathsf{Fi}({#1})}
\newcommand{\linear}[1]{\mathsf{H}(3,{#1}^2)}
\newcommand{\pentmon}[1]{\mathsf{PM}({#1})}
\newcommand{\tr}{\mathrm{Tr}}
\newcommand\half{\tfrac{1}{2}}

\renewcommand{\le}{\leqslant}
\renewcommand{\ge}{\geqslant}

\newcommand{\bvec}[1]{\boldsymbol{#1}}
\newcommand{\allones}{\boldsymbol{j}}

\begin{document}

\begin{frontmatter}

\title{Hemisystems of small flock \gqs{}}
 
\author[uwa]{John Bamberg}
\ead{John.Bamberg@uwa.edu.au}

\author[uwa]{Michael Giudici}
\ead{Michael.Giudici@uwa.edu.au}

\author[uwa]{Gordon F. Royle}
\ead{Gordon.Royle@uwa.edu.au}

\address[uwa]{
   The Centre for the Mathematics of Symmetry and Computation,\\
   School of Mathematics and Statistics,\\
   The University of Western Australia,\\
   35 Stirling Highway, Crawley, W. A. 6009\\
  Australia}



\begin{abstract}
In this paper, we describe a complete computer classification of the hemisystems in the two known flock \gqs{} of order $(5^2,5)$ and give numerous further examples of hemisystems in all the known flock \gqs{} of order $(s^2,s)$ for $s \le 11$. By analysing the computational data, we identify two possible new infinite families of hemisystems in the classical \gq{} $\linear{s}$.
\end{abstract}

\end{frontmatter}

\section{Introduction}\label{section:intro}

A \textit{hemisystem of lines} of a \gq{} of order $(s^2,s)$ is a set $\mathcal{H}$ of lines such that every point $P$ is incident with $(s+1)/2$ elements of $\mathcal{H}$; that is, exactly half of the lines incident with each point lie in $\mathcal{H}$. The complementary set of lines to a hemisystem is also a hemisystem that may or may not be equivalent under the automorphism group of the \gq{} --- if it is equivalent to its complement then we call it {\em self-complementary}.  Hemisystems give rise to various other combinatorial objects, including partial quadrangles (Cameron \cite{Cameron75}), strongly regular graphs with certain parameters, and $4$-class imprimitive cometric $Q$-antipodal association schemes\footnote{In fact, these cometric association schemes have Krein array $\{(q^2+1)(q+1),(q^2-q+1)^2/q,(q^2-q+1)(q-1)/q,1;1,(q^2-q+1)(q-1)/q,(q^2-q+1)^2/q,(q^2+1)(q-1)\}$.} that are not metric (see van Dam, Martin and Muzychuk  \cite{MartinMuzychukvanDam}), all of which were thought to be somewhat rare.

The notion of a hemisystem was introduced in 1965 by Segre \cite{Segre65} in his work on \textit{regular systems} of the Hermitian surface, and he proved that  there is a unique hemisystem of lines (up to equivalence) of the classical \gq{} $\linear{3}$. It was long thought that this was the only hemisystem in $\linear{q}$ and indeed Thas \cite{Thas95} conjectured this as late as 1995. However, forty years after Segre's seminal paper, Cossidente and Penttila \cite{CossidentePenttila05} constructed an infinite family of hemisystems of the classical quadrangles $\linear{q}$ and other authors subsequently constructed sporadic examples in $\linear{q}$ \cite{BKLP07,CossidentePenttila09} and a single example in the non-classical \gq{} $\ftwkb{5}$ (see \cite{BambergDeClerckDurante09}). The first main result of this paper extends the complete classification of hemisystems to the (known) \gqs{} of order $(5^2,5)$.

\begin{theorem}
\label{thm:q=5} 
A hemisystem of the classical \gq{} $\linear{5}$ is equivalent to one of the two self-complementary hemisystems described in Table~\ref{tab:H35} and a hemisystem of the Fisher-Thas-Walker-Kantor-Betten \gq{} $\ftwkb{5}$ is equivalent to one of the three complementary pairs described in Table~\ref{tab:ftwkb5}. 
\end{theorem}

All known \gqs{} of order $(s^2,s)$, $s$ odd, arise from flocks of the quadratic cone and hence are called \textit{flock \gqs{}}. In \cite{BambergGiudiciRoyle1} we gave a general construction for hemisystems that produces a hemisystem in every flock \gq{}, known or unknown. In fact (as pointed out to us by Tim Penttila), our construction shows that the number of hemisystems in any infinite family of flock \gqs{} grows exponentially with the size of the \gq{}. Therefore, far from being rare, hemisystems and their associated partial quadrangles, strongly regular graphs etc. actually exist in great profusion.
Of course this is an asymptotic result only, and so in this companion paper to \cite{BambergGiudiciRoyle1}, we consider hemisystems in the small (known) flock \gqs{}, namely those of order $(s^2,s)$ for all (odd) $s \le 11$. Using a mixture of computation and analysis driven by the computational data, we discover large numbers of new hemisystems that do not arise from our general construction. Apart from the smallest
\gqs{}, our searches all assume the existence of some group of symmetries stabilising the hemisystem and so are necessarily incomplete.

Table~\ref{knownhemis} summarises the results of our investigations, dividing the hemisystems into those of Type I arising from construction of \cite{BambergGiudiciRoyle1} which we review in Section~\ref{section:construction}, and those that do not arise from this construction. In this table, notation of the form $6 \times 2 + 2$ is used to indicate that, up to equivalence under the automorphism group of the \gq{}, there are 6  complementary pairs of hemisystems and 2 self-complementary hemisystems, for a total of 14 hemisystems.  In Theorem~\ref{thm:stab} we show that a hemisystem of Type I in a \gq{} of order $(q^2,q)$ is invariant under an elementary abelian group of order $q^2$, so one way to verify that a hemisystem is not of Type I is to show that it is not invariant under such a group.

By analysing the computational data for the classical \gqs{} $\linear{q}$, we identify patterns that suggest the existence of three possible new infinite families of 
hemisystems. For these candidate families, we extend the computations to higher values of $q$ and, based on these computations, conjecture that just two of the three candidate families continue indefinitely. These families are discussed in Section~\ref{sec:inffamilies}.

We end the paper in Section \ref{sec:probs} by discussing a number of questions and directions for future research suggested by our results.

\renewcommand{\arraystretch}{1.2}

\begin{table}
\begin{center}
\begin{tabular}{|c|c|c|c|c|}
\hline
$q$&GQ&Type I hemisystems&Other hemisystems&Total\\
\hline\hline
$3$&$\linear3$&$1$&$0$&$1$\\
\hline
$5$&$\linear5$&$2$&$0$&$2$\\
&$\ftwkb5$&$1\times 2$&$2 \times 2$&$6$\\
\hline
$7$&$\linear7$&$2$&$4$&$6$\\
&$\kmonom7$&$6 \times 2 + 2$&$6 \times 2+2$&$28$\\
\hline
$9$&$\linear9$&$3$&$4$&$7$\\
&$\kknuth9$&$3$&$2$&$5$\\
&$\fish9$&$6 \times 2 + 9$&$4 \times 2 + 5$&$34$\\
\hline
$11$&$\linear{11}$&$6$&$1 \times 2 + 5$&$13$\\
&$\ftwkb{11}$&$10 \times 2$&&$20$\\
&$\fish{11}$&$42 \times 2 + 6$&$6 \times 2$&$102$\\
&$\pentmon{11}$&$ 74 \times 2 + 8$&$18 \times 2$&$192$\\
\hline
\end{tabular}
\end{center}
\caption{Known hemisystems in the flock \gqs{} of order $(s^2,s)$ for $s \le 11$}
\label{knownhemis}
\end{table}

\section{Some basic background theory}\label{section:background}

A {\em \gq{}} is an incidence structure of points and
lines such that if $P$ is a point and $\ell$ is a line not incident with
$P$, then there is a unique line through $P$ which meets $\ell$ in a
point. From this property, in the finite case, if there is a line containing at least three points or if there
is a point on at least three lines, then there are constants $s$ and $t$
such that each line is incident with $s+1$ points, and each point is
incident with $t+1$ lines. Such a \gq{} is said to have
\emph{order} $(s,t)$, and its point-line dual is a \gq{} of order $(t,s)$.

In this paper we are concerned with \gqs{} of order $(s^2,s)$, for $s$ odd. The classical example is the incidence structure of all points and lines of a non-singular
Hermitian variety in $\PG(3,q^2)$, which forms the \textit{classical} {\gq}
$\linear{q}$ of order $(q^2,q)$ (see \cite[3.2.3]{FGQ}). Further examples can be constructed from BLT-sets using the Knarr model. We briefly outline this construction below.

\subsection{Flocks of quadratic cones and BLT-sets}\label{flocks}

A \textit{flock} of the quadratic cone $\mathcal{C}$ with vertex $v$ in $\PG(3,q)$ is a
partition of the points of $\mathcal{C}\backslash\{v\}$ into conics. J. A. Thas \cite{Thas87}
showed that a flock gives rise to an elation \gq{} of order $(q^2,q)$,
which we call a \textit{flock quadrangle}.  A \textit{BLT-set of lines} of
$\mathsf{W}(3,q)$ is a set $\mathcal{O}$ of $q+1$ lines of $\mathsf{W}(3,q)$ such that no
line of $\mathsf{W}(3,q)$ is concurrent with more than two lines of $\mathcal{O}$.  
In \cite{BLT}, it was shown that, for $q$ odd, a flock of a quadratic cone in $\PG(3,q)$
gives rise to a BLT-set of lines of $\mathsf{W}(3,q)$. Conversely, a BLT-set gives rise to
possibly many flocks, however we only obtain one flock quadrangle up to isomorphism
(see \cite{PayneRogers90}). 

For $q$ odd, Knarr
\cite{Knarr92} gave a direct geometric construction of a flock quadrangle from a BLT-set
of lines of $\mathsf{W}(3,q)$.   Applying this construction to a \textit{linear} BLT-set of
lines (i.e., a \textit{regulus} obtained from field reduction of a Baer subline) of $\mathsf{W}(3,q)$, yields a \gq{} isomorphic to the classical object $\linear{q}$.

The BLT-sets of lines of $\mathsf{W}(3,q)$ have been
classified by Law and Penttila \cite{LawPenttila03} for prime powers $q$ at most $29$, and this has recently been extended by Betten \cite{Betten} to $q\le 67$. We outline the main infinite families in Section \ref{qclans}.

\subsection{The Knarr model}

The symplectic polar space $\mathsf{W}(5,q)$ of rank $3$ is the geometry arising from
taking the one-, two- and three-dimensional vector subspaces of $\GF(q)^6$ for which a given
alternating bilinear form restricts to the zero form (i.e., the \textit{totally isotropic}
subspaces). For example, one can take this alternating bilinear form to be defined by
$$\beta(\bvec{x} , \bvec{y} ) = x_1y_6-x_6y_1+x_2y_5-x_5y_2+x_3y_4-x_4y_3.$$ 
In particular $\beta(\bvec{x},\bvec{y})=\bvec{x}J\bvec{y}^T$ where 
{\small $$J=\left(\begin{array}{rrrrrr}
       0&0&0&0&0&1\\
       0&0&0&0&1&0\\
       0&0&0&1&0&0\\
       0&0&-1&0&0&0\\
       0&-1&0&0&0&0\\
       -1&0&0&0&0&0
\end{array}\right)$$}
This bilinear form
also determines a null polarity $\perp$ of the ambient projective space $\PG(5,q)$, defined by
$U\mapsto U^\perp := \{\bvec{v}\in \GF(q)^6: \beta( \bvec{u},\bvec{v} ) =0\text{ for all } \bvec{u} \in U\}$. 

The ingredients of the Knarr construction are as follows:
\begin{itemize}
\item a null polarity $\perp$ of $\mathsf{PG}(5,q)$;
\item a point $P$ of $\mathsf{PG}(5,q)$;
\item a BLT-set of lines $\mathcal{O}$ of $\mathsf{W}(3, q)$.
\end{itemize}
Note that the totally isotropic lines and planes incident with $P$ yield the quotient polar space $P^\perp/P$
isomorphic to $\mathsf{W}(3, q)$.  So we will abuse notation
and identify $\mathcal{O}$ with a set of totally isotropic planes on $P$.  Then we construct a
\gq{} $\mathcal{K}(\mathcal{O})$ as in Table \ref{tab:flockgq}.

\begin{table}[ht]
\begin{tabular}{lp{6.5cm}|lp{6.5cm}}
&Points && Lines \\
\hline
(i) &points of $\mathsf{PG}(5,q)$ not in $P^\perp$&(a)& totally isotropic planes not contained in $P^\perp$ and meeting some element of $\mathcal{O}$ in a line \\
(ii) &lines not incident with $P$ but contained in some element of $\mathcal{O}$&(b)& elements of $\mathcal{O}$\\
(iii)& the point $P$&&\\
\hline \\
\end{tabular}
\medskip
 Incidence is inherited from that of $\mathsf{PG}(5,q)$.
\caption{Knarr model for a flock \gq{}}
\label{tab:flockgq}
\end{table}

We now describe how the Knarr model leads to some obvious automorphisms of the resulting \gq{} $\mathcal{K}(\mathcal{O})$.
Let $G$ be the semisimilarity group of the form $\beta$, that is, the group of all semilinear transformations $g$ of $\GF(q)^6$ for which there exists $\lambda\in\GF(q)$ and $\sigma\in\Aut(\GF(q))$ such that $\beta(\bvec{u}^g,\bvec{v}^g)=\lambda\beta(\bvec{u}, \bvec{v})^{\sigma}$ for all $\bvec{u}, \bvec{v}\in\GF(q)^6$. Let $H$ be the group of similarities of $\beta$, that is, the group of all linear transformations that preserve $\beta$ up to a scalar. Then
$$H=\{A\in\GL(6,q)\mid AJA^T=\lambda J \text{ for some }\lambda\in\GF(q)\}\cong \GSp(6,q).$$
 Let $$J'=\begin{pmatrix}
         0&0&0&1\\
       0&0&1&0\\
       0&-1&0&0\\
       -1&0&0&0\\
 \end{pmatrix}$$ and take $P$ to be the span of $[1,0,0,0,0,0]$. Then $H_P=E\rtimes (Q\times R)$ where 
$$\begin{array}{rl}
    E &=\left\{\begin{pmatrix}
             1&0&0\\
            (J')^T\bvec{a}^T&I&0\\
            z&\bvec{a}&1
           \end{pmatrix}\Big\vert \,\, \bvec{a}\in\GF(q)^4,z\in\GF(q)\right\}\\
Q &=\left\{\begin{pmatrix}
             \lambda&0&0\\
            0&I&0\\
            0&0&\lambda^{-1}
           \end{pmatrix}\Big\vert \,\, \lambda\in\GF(q)\backslash\{0\}\right\}\cong C_{q-1}\\
R &=\left\{\begin{pmatrix}
             \lambda&0&0\\
            0&A&0\\
            0&0&1
           \end{pmatrix}\Big\vert \,\, A\in\GL(4,q), AJ'A^T=\lambda J'\right\}\cong\GSp(4,q) \end{array}$$
and $(H_P)_{\mathcal{O}}=E\rtimes (Q\times R_{\mathcal{O}})\cong E\rtimes (Q\times \GSp(4,q)_{\mathcal{O}})$. Moreover, $G_P=\langle H_P,\sigma\rangle$, where $\sigma$ is the standard Frobenius map. Note that $\langle R,\sigma\rangle \cong\Gamma\Sp(4,q)$ and acts on $E/Z(E)$ as in its natural action on a 4-dimensional vector-space over $\GF(q)$. Moreover, $(G_P)_{\mathcal{O}}=E\rtimes (Q\rtimes \langle R,\sigma\rangle_{\mathcal{O}})\cong E\rtimes (Q\rtimes \Gamma\Sp(4,q)_{\mathcal{O}})$.  The group $(G_P)_{\mathcal{O}}$ preserves the flock \gq{} $\mathcal{K}(\mathcal{O})$ and contains the subgroup $Z$ of all scalar matrices. Hence $E\rtimes \Gamma\Sp(4,q)_{\mathcal{O}}\cong (G_P)_{\mathcal{O}}/Z\leqslant\Aut(\mathcal{K}(\mathcal{O}))$. In fact,  if the flock quadrangle $\mathcal{K}(\mathcal{O})$ is not classical, then these are the only automorphisms that you get, that is, $\Aut(\mathcal{K}(\mathcal{O}))=E\rtimes \Gamma\Sp(4,q)_{\mathcal{O}}$ \cite[IV.1 and IV.2]{paynethasflockauto}. (Note: In the paper \cite{BambergGiudiciRoyle1} we incorrectly claimed that additional automorphisms could arise for the Kantor-Knuth \gqs{}.)

\section{Hemisystems of Type I and their automorphisms}\label{section:construction}

In this section we revise the construction given in \cite{BambergGiudiciRoyle1} and discuss the stabiliser of the resulting hemisystems.

\begin{lemma}[{Bamberg, Giudici and Royle \cite{BambergGiudiciRoyle1}}]\label{eqrelation}
Consider a set $\mathcal{O}$ of totally isotropic planes of $\mathsf{W}(5,q)$ each incident with a point $P$ such
that $\{\pi/P:\pi\in\mathcal{O}\}$ is a BLT-set of lines of the quotient symplectic space $P^\perp/P\cong\mathsf{W}(3,q)$. Let $\ell$ be a line of $\mathsf{W}(3,q)$ not meeting any element of $\mathcal{O}$. Define a binary relation $\equiv_\ell$ on $\mathcal{O}$ 
by setting $\pi\equiv_\ell \pi'$ if and only if
$$\pi=\pi'\quad\text{  or   }\quad\{\langle Y,Y^\perp\cap\pi\rangle\mid Y\in\ell\} \cap\{\langle Y,Y^\perp\cap\pi'\rangle\mid Y\in\ell\}=\varnothing.$$
Then $\equiv_\ell$ is an equivalence relation yielding a partition of $\mathcal{O}$ into two
parts of equal size.
\end{lemma}

\begin{theorem}[Bamberg, Giudici and Royle \cite{BambergGiudiciRoyle1}]\label{construction}
Consider a set $\mathcal{O}$ of totally isotropic planes of $\mathsf{W}(5,q)$ each incident with a point $P$ such
that $\{\pi/P:\pi\in\mathcal{O}\}$
is a BLT-set of lines of the quotient symplectic space $P^\perp/P\cong\mathsf{W}(3,q)$. 
Suppose that we have a line $\ell$ of $\mathsf{W}(5,q)$ not meeting
any element of $\mathcal{O}$, and let $\equiv_\ell$ be the binary relation on $\mathcal{O}$ defined in Lemma \ref{eqrelation} with equivalence classes
$\mathcal{O}^+$ and $\mathcal{O}^-$.
Let $\mathcal{S}$ be a subset of the totally isotropic planes 
on $\ell$ of size $(q-1)/2$, not containing $\langle P,\ell\rangle$, and let $\mathcal{S}^c$
be the complementary set of planes on $\ell$. 
Let 
\begin{enumerate}
\item[(i)] $\mathcal{L}^{+}_{\mathcal{S}}$ be the totally isotropic planes that meet some element of
$\mathcal{O}^{+}$ in a line, and which meet some element of $\mathcal{S}$ in a point; and
\item[(ii)] $\mathcal{L}^{-}_{\mathcal{S}^c}$ be the totally isotropic planes that meet some element of
$\mathcal{O}^{-}$ in a line, and which meet some element of $\mathcal{S}^c$ in a point;
\end{enumerate}
Then $\mathcal{O}^+\cup \mathcal{L}^+_{\mathcal{S}}\cup \mathcal{L}^-_{\mathcal{S}^c}$
is a hemisystem of lines of $\mathcal{K}(\mathcal{O})$.
\end{theorem}

Recall that Cossidente and Penttila showed that for each odd $q$, there exists a hemisystem $\mathcal{H}_q$ of $\mathsf{H}(3,q^2)$ admitting
$\mathsf{P\Omega}^-(4,q)$. It was shown in \cite{BambergGiudiciRoyle1} that these hemisystems could be constructed using Theorem \ref{construction}.
Moreover, the number of hemisystems produced by this construction grows exponentially with $q$. 
To see this, note that the number of choices of $(q-1)/2$ things from $q$ things is the binomial coefficient; asymptotically this has value
$\frac{2^{q}\sqrt {2 / \pi}}{\sqrt{q+1}}$, or basically, $\Theta(2^q / \sqrt{q})$. Whereas the automorphism group of the \gq{} is polynomial in size and hence there are exponentially many inequivalent choices.

\begin{theorem}
\label{thm:stab}
Let $\mathcal{H}$ be the hemisystem exhibited in Theorem \ref{construction} and let $G$ be the automorphism group of the \gq{} $\mathcal{K}(\mathcal{O})$. Then $G_{\mathcal{H}}$ contains $T\rtimes \Sp(4,q)_{\mathcal{O}^+,\mathcal{O}^-,\ell'}$, where $T$ is an elementary abelian group of order $q^2$ and $\ell'$ is the line of $\mathsf{W}(3,q)$ obtained by projecting $\ell$ onto $P^\perp/P$. The group $T$ acts semiregularly on the set of lines of type (a) of $\mathcal{K}(\mathcal{O})$ and fixes each line of type (b).
\end{theorem}
\begin{proof}
Consider the group
$$ E=\left\{\begin{pmatrix}
             1&0&0\\
            J^T\bvec{a}^T&I&0\\
            z&\bvec{a}&1
           \end{pmatrix} \Big\vert \,\, \bvec{a}\in\GF(q)^4,z\in\GF(q) \right\}$$
which acts on the \gq{} $\mathcal{K}(\mathcal{O})$.

Let $\mathcal{O}$ be our BLT-set, considered as a set of lines in $\mathsf{W}(3,q)$. Each $\langle \bvec{u}_1,\bvec{u}_2\rangle\in\mathcal{O}$ is identified with the 3-space $\langle P, [0,\bvec{u}_1,0],[0,\bvec{u}_2,0]\rangle$ in $V$.  Note that 
$$[0,\bvec{u},0]\begin{pmatrix}
        1&0&0\\
            J^T\bvec{a}^T&I&0\\
            z&\bvec{a}&1
\end{pmatrix} =[\bvec{u}J^T\bvec{a}^T,\bvec{u},0]$$
Hence $E$ fixes each plane on $P$ and hence each element of $\mathcal{O}$.  Moreover, given a line $\ell$ in $P^{\perp}$ that is disjoint from every element of $\mathcal{O}$, we have that $E$ fixes $\langle P,\ell\rangle$. Now $\langle P,\ell\rangle$ contains $q^2$ lines not on $P$. If we take $\ell'=\langle [0,\bvec{w}_1,0],[0,\bvec{w}_2,0]\rangle$ to be a line on $\langle \ell,P\rangle$ we see that 
$$E_{\ell'}=\left\{\begin{pmatrix}
      1&0&0\\
            J^T\bvec{a}^T&I&0\\
            z&\bvec{a}&1
\end{pmatrix} \Big\vert \,\, z\in\GF(q),\bvec{w}_1J^T\bvec{a}^T=\bvec{w}_2J^T\bvec{a}^T=0\right\}$$
which has order $q^3$. Thus $E$ acts transitively on the set of lines of $\langle P,\ell\rangle$ not on $P$ and so we may choose $\ell=\langle [0,\bvec{w}_1,0],[0,\bvec{w}_2,0]\rangle$. We let $\ell'=\langle \mathbf{w}_1,\mathbf{w}_2\rangle$, a totally isotropic line in $\mathsf{W}(3,q)$.

Let $\mathcal{R}$ be the set of totally isotropic planes on $\ell$ other than $\langle P,\ell\rangle$. Note that $E_{\ell}$ fixes $\mathcal{R}$ setwise. These planes are of the form $\langle \ell,[x_1,0,0,0,0,1]\rangle$ with $x_1\in\GF(q)$.  Let $T$ be the elementary abelian subgroup of $E_{\ell}$ of order $q^2$ consisting of all elements with $z=0$. Then 
$$[x_1,0,0,0,0,1]\begin{pmatrix}
       1&0&0\\
            J^T\bvec{a}^T&I&0\\
            0&\bvec{a}&1
\end{pmatrix} =[x_1,\bvec{a},1]$$
Since $\bvec{a}\in \langle \bvec{w}_1,\bvec{w}_2\rangle^{\perp}=\langle \bvec{w}_1,\bvec{w}_2\rangle$, it follows that $[x_1,\bvec{a},1]\in\langle \ell,[x_1,0,0,0,0,1]\rangle $ and so $T$ fixes each element of $\mathcal{R}$. 

Let $\mathcal{S}$ be a subset of size $(q-1)/2$ of $\mathcal{R}$ and $\mathcal{S}^c$ be the complementary set of totally isotropic planes of size $(q+1)/2$. Then $T$ fixes $\mathcal{S}$ and $\mathcal{S}^c$ elementwise. Hence $T$ fixes the hemisystem $\mathcal{H}=\mathcal{O}^+\cup\mathcal{L}_{\mathcal{S}}^+\cup \mathcal{L}_{\mathcal{S}^c}^-$.

Let $B\in\Sp(4,q)_{\mathcal{O}}$ and consider the element
$$X=\begin{pmatrix}
 1 & 0_{1\times 4} &0\\
 0_{4\times 1}&B&0_{4\times 1}\\
 0 &0_{1\times 4} &1
\end{pmatrix}$$
which acts on the flock \gq{} $\mathcal{K}(\mathcal{O})$. If $B\in \Sp(4,q)_{\mathcal{O}^+,\mathcal{O}^-,\ell'}$ then $X$ fixes each element of $\mathcal{R}$ setwise and hence stabilises the hemisystem $\mathcal{O}^+\cup\mathcal{L}_{\mathcal{S}}^+\cup \mathcal{L}_{\mathcal{S}^c}^-$. Thus $T\rtimes \Sp(4,q)_{\mathcal{O}^+,\mathcal{O}^-,\ell'}\leqslant G_{\mathcal{H}}$.

The lines of $\mathcal{K}(\mathcal{O})$ are the elements of $\mathcal{O}$ and the totally isotropic planes not on $P$ and meeting some element of $\mathcal{O}$ in a line. We have seen already that $T$ fixes each of the elements of $\mathcal{O}$. Now let $U=\langle P,[0,\mathbf{u}_1,0],[0,\mathbf{u}_2,0]\rangle\in\mathcal{O}$ and recall that $\langle \mathbf{u}_1,\mathbf{u}_2\rangle\cap \langle \bvec{w}_1,\bvec{w}_2\rangle=\{0\}$. Then
$$T_{\langle [0,\mathbf{u}_1,0],[0,\mathbf{u}_2,0]\rangle}=\left\{\begin{pmatrix}
       1&0&0\\
            J^T\bvec{a}^T&I&0\\
            0&\bvec{a}&1
\end{pmatrix}\in T\Big\vert \,\, \bvec{u}_1J^T\bvec{a}^T=\bvec{u}_2J^T\bvec{a}^T=0\right\}$$
Since such elements lie in $T$ they also satisfy $\bvec{w}_1J^T\bvec{a}^T=\bvec{w}_2J^T\bvec{a}^T=0$. If $\bvec{a}\neq \bvec{0}$, we have $\{\bvec{x}\mid \bvec{x}J^T\bvec{a}^T=0\}$ has dimension 3 but contains the complementary 2-spaces $\langle \mathbf{u}_1,\mathbf{u}_2\rangle$ and $\langle \bvec{w}_1,\bvec{w}_2\rangle$. This is a contradiction and so $T_{\langle [0,\mathbf{u}_1,0],[0,\mathbf{u}_2,0]\rangle}=1$. Thus $T$ acts regularly on the $q^2$ lines in $U$ not containing $P$, and hence acts semiregularly on the totally isotropic planes not on $P$ and meeting some element of $\mathcal{O}$ in a line. 
\qed\end{proof}

\begin{remark}
The stabiliser $G_{\mathcal{H}}$ can be larger than the group given by Theorem \ref{thm:stab}. Sometimes extra automorphisms can arise from the structure of the Knarr model. For example, if $\mathcal{S}$ were chosen to be $\{\langle \ell,[x^2,0,0,0,0,1]\rangle \mid x\in\GF(q)\}$ then the elements 
$$\begin{pmatrix}
 \lambda & 0_{1\times 4} &0\\
 0_{4\times 1}&I_{4\times 4}&0_{4\times 1}\\
 0 &0_{1\times 4} &\lambda^{-1}
\end{pmatrix}$$ will fix $\mathcal{H}$. Similarly, suitable choices of $\mathcal{S}$ may give rise to semisimilarities of $\beta$ that stabilise $\mathcal{H}$. 

Alternatively, $\linear{q}$ has more automorphisms than those arising from the Knarr model. The Cossidente-Penttila hemisystems in these \gqs{} admit at least $\mathsf{P}\Sigma\mathsf{L}(2,q^2)$.
\end{remark}

\section{Potential new infinite families of hemisystems of $\linear{q}$}
\label{sec:inffamilies}

Examination of our computational data uncovered three promising candidates for new infinite families of hemisystems of $\mathsf{H}(3,q^2)$, and in this section we investigate these possible families in more detail.

\subsection{Hemisystems that are invariant under a Singer type element}
\label{sec:metacyclic}

In this section, we present a way of viewing hemisystems of $\mathsf{H}(3,q^2)$ that are invariant under a \textit{Singer type element},
and we give some computational data which shows the existence of such hemisystems for all $q\le 29$, except (curiously) $q\in\{13,25\}$. For $q=3$ we obtain the Segre hemisystem, for $q=5$ we obtain the hemisystem invariant under $(3\cdot A_7).2$ discovered by Cossidente and Penttila \cite{CossidentePenttila05} and for $q=7,9$ the examples are given in \cite{BKLP07}. In what follows, we will work in the dual \gq{}; the points and lines of the elliptic quadric $\mathsf{Q}^-(5,q)$. A hemisystem of
lines of $\mathsf{H}(3,q^2)$ transfers to a \textit{hemisystem of points} of $\mathsf{Q}^-(5,q)$.

We begin with $\GF(q^6)$ and equip it with the following bilinear form over $\GF(q)$:
$$B(x,y) := \tr_{q^6\to q}(xy^{q^3}).$$
(Note that $ \tr_{q^6\to q}$ is the relative trace map $x\mapsto x+x^q+x^{q^2}+x^{q^3}+x^{q^4}+x^{q^5}$).
This form is symmetric and defines an elliptic orthogonal space isomorphic to $\mathsf{Q}^-(5,q)$.
Now let $\omega=\xi^{(q^3-1)(q+1)}$ where $\xi$ is a primitive element of $\GF(q^6)$. Let $K=\langle \omega\rangle$ and note that $K$
is independent of the choice of $\xi$ (it is just the set of elements $x$ such that $x^{q^2-q+1}=1$). 
Then $K$ is irreducible and acts semiregularly on the totally isotropic points of $\mathsf{Q}^-(5,q)$, and is occasionally known as
a \textit{Singer type} isometry of $\mathsf{Q}^-(5,q)$.
So the number of orbits of $K$ on totally isotropic points is $(q+1)^2$. It is not difficult to see that each point orbit is of the form
$$\{\langle u\rangle \mid u^{(q^2-q+1)(q-1)}=r\},$$
where $r$ is a singular element of $\GF(q^6)^*$ such that $r^{(q+1)}\in\GF(q^3)$.
In what follows, we will use the underlying vectors instead of the projective points as the equations will be simpler.
Note that the $K$-orbits on singular nonzero vectors are each of the form
$$ \{ u\in\GF(q^6)^* \mid u^{q^2-q+1}=r\},\quad r\in R$$
where
$$R:=\{r\in\GF(q^6)\mid r^{q+1}\in\GF(q^3), \tr_{q^6\to q}(r^{q+1})=0\}.$$
The elements of $R$ lie on the mutually disjoint lines
$$\ell_a:X^{q^2}-aX=0$$
where $a$ is an element of $\GF(q^3)$ such that $a^{q+1}+a+1=0$.

So to construct a hemisystem, we need to construct a set of $\half(q+1)^2$ elements of $R$. Of the hemisystems we found, all were invariant under the field automorphism $\tau:a\mapsto a^{q^2}$ fixing $\GF(q^2)$ elementwise, and it acts on the set of 
lines $\{\ell_a\}$. The orbits of $\langle \tau\rangle$ on $\{\ell_a\}$ are the zero sets of the $\GF(q^2)$-irreducible factors of the polynomial $X^{q+1}+X+1$.
Now every element $r\in R$ can be uniquely represented by the pair $(r^{q^2-1},r^{q^3-1})$. The possible values of $r^{q^2-1}$ are the $q+1$ zeros of $X^{q+1}+X+1$,
and the possible values of $r^{q^3-1}$ are the $q+1$ solutions to $X^{q+1}=1$; let this latter set be denoted by $N$. 
So a  $\langle \tau\rangle$-orbit on $R$ is uniquely determined by a $\GF(q^2)$-irreducible factor $i(X)$ of $X^{q+1}+X+1$ and an element $n\in N$:
$$\{ r\in R: i(r^{q^2-1})=0, r^{q^3-1}=n\}.$$
The hemisystems we construct arise from unions of these orbits. 

Below we list the hemisystems that we found for $3\le q \le 9$. In each table we describe each solution by 
unions of $\langle \tau\rangle$-orbits on $R$. The constituents of these unions are described by
which values of $N$ appear as right-hand values for each $i(X)$.

\begin{example}\hrule\medskip
For $q=3$, $N=\{1,-1,z^2,z^6\}$, where $z$ is the primitive root of $\GF(q^2)$. The $\GF(q^2)$-irreducible factors of $X^{q+1}+X+1$ are
$$i_1(X):X-1\quad\text{ and }\quad i_2(X):X^3+X^2+X-1.$$ Let $\Pi$ be the subset of the ordered pairs $\{i_1,i_2\}\times N$ described by specifying the right-hand coordinates
per possible left-hand coordinate:
\begin{center}
\begin{tabular}{|c|c|}
\hline
$X-1$& $1, z^6$ \\
  $X^3+X^2+X-1$&$-1, z^2$ \\
  \hline
\end{tabular}
\end{center}
Now let
$$\mathcal{H}^R_\Pi:=\{ r\in R \mid i(r^{q^2-1})=0, (i(X), r^{q^3-1})\in\Pi\}.$$
Then our hemisystem of points of $\mathsf{Q}^-(5,q)$ is simply
$$\mathcal{H}_\Pi:=\left\{\langle u\rangle\mid u^{(q^2-q+1)(q-1)}\in \mathcal{H}^R_\Pi \right\}.$$ 
Moreover, we know that $\mathcal{H}_\Pi$ is projectively equivalent to the Segre hemisystem.\medskip\hrule
\end{example}

In each case below, $z$ denotes the primitive element of $\GF(q^2)$.
For each $q$ below, we list one solution, and all the solutions can be obtained by taking the given solution and its orbit under the action of $\langle z\rangle$.

\begin{table}[H]
\begin{center}\small{
\begin{tabular}{c|c|c}
$q$& $i(X)$ & $N$ \\
\hline
$3$  &$X-1$& $1, z^6$ \\
&  $X^3+X^2+X-1$&$-1, z^2$ \\
\hline
$5$ & $X^3+2X^2-X-1$ & $1 , z^8 ,  z^{16}$\\
& $X^3+3X^2-1$ & $1 ,  z^4 ,  z^{20}$\\
\hline
$7$ & $X+3$ & $z^{6} ,   z^{12} ,   z^{30} ,   z^{36} $  \\
& $X+5$ &    $ 1 ,   -1 ,   z^{18} ,   z^{42}$ \\
& $ X^3+4X-1$ & $1,-1,z^{18} ,   z^{42}$ \\
& $X^3-X^2+3X-1$& $ z^{6} ,   z^{12} ,  z^{30} ,   z^{36}$  \\
\hline
$9$ & $X-1$& $ 1, z^{8}, z^{24}, z^{56}, z^{72}$\\
&  $X^3-X^2-X-1$ & $ 1, z^{16}, z^{32}, z^{48}, z^{64}$  \\
& $X^3+z^{50}X^2+z^{50}X-1 $ & $ 1, z^{8}, z^{16}, z^{64}, z^{72}$\\
&  $X^3+z^{70}X^2+z^{70}X-1$ & $ 1, z^{24}, z^{32}, z^{48}, z^{56}$\\
\hline
\end{tabular}}
\caption{Sets $\Pi$ of ordered pairs $(i(X), r^{q^3-1})$.}
\end{center}
\end{table}

We have found hemisystems for larger $q$ and we summarise them below.

\begin{center}

\begin{tabular}{c|c|c}
$q$&$q^2-q+1$&Stabiliser\\
\hline
3&7&$\mathrm{PSL}(3,4). 2$\\
5&21&$3\cdot A_7\cdot 2$\\
7&43&$43: 6$\\
9&73&$73:6$\\
11&111&$111:6$, $333:3$\\
17&273&$273: 3$\\
19&1715&$1715: 6$\\
23&507&$507: 6$\\
27&703&at least $703: 3$\\
\hline
\end{tabular}
\end{center}

\begin{problem}
\label{prob:metacyclic}
Does there exist a hemisystem invariant under a Singer type element for all odd prime powers $q\not\equiv 1\pmod{12}$?
\end{problem}

\subsection{Hemisystems invariant under the stabiliser of a triangle: tyranny of the small?}
\label{sec:triangular}

Another interesting sequence of hemisystems apparent in  our data is that for $q=7,9$ and 11, the \gq{} $\linear{q}$ contains a hemisystem invariant under a group  $K=C_{q+1}^2:S_3$. In fact, for $q=9$ and $11$ there are several such hemisystems. Moreover, the stabiliser of the Segre hemisystem for $q=3$ contains such a subgroup, as does the group $(3\cdot A_7).2$ for $q=5$.  

The group $K$ can be realised as follows. The stabiliser of a nondegenerate hyperplane of $\linear{q}$ contains a group $H\cong C_{q+1}^3:(S_3\times C_{2f})$ where $q=p^f$ that fixes a set $T$ of mutually orthogonal nondegenerate points $\{\langle v_1\rangle,\langle v_2\rangle,\langle v_3\rangle\}$ of the underling projective space. In particular, taking $v_1,v_2,v_3$ as the first three elements of a basis of the underlying vector space, the pointwise stabiliser in $\PGU(4,q)$ of $T$ is the group $D$ of all diagonal matrices $\mathrm{diag}(\lambda_1,\lambda_2,\lambda_3,1)$ such that $\lambda_1^{q+1}=\lambda_2^{q+1}=\lambda_3^{q+1}=1$. Letting $\sigma$ and $\tau$ be the permutation matrices such that $\sigma:v_1\mapsto v_2\mapsto v_3\mapsto v_1$ and $\tau:v_1\mapsto v_1, v_2\mapsto v_3\mapsto v_1$, we have $\langle\sigma,\tau\rangle\cong S_3$. Moreover, $H=D: (\langle \sigma,\tau\rangle \times \langle\phi \rangle)$ where $\phi$ is the field automorphism such that  $\phi:\sum\lambda_iv_i \mapsto \sum\lambda_i^pv_i$. The group $H$ contains a normal subgroup $R$ isomorphic to $C_{q+1}^2$ given by 
$$R:=\{\mathrm{diag}(\lambda_1,\lambda_2,\lambda_3,1)\mid \lambda_i^{q+1}=1,\lambda_1\lambda_2\lambda_3=1\}.$$
The group $K$ that leaves invariant a hemisystem for the values of $q$ examined is $R\rtimes \langle \sigma,\mathrm{diag}(\lambda,\lambda,\lambda,1)\tau\phi^f\rangle$ where $\lambda$ is an element of order $q+1$.

So naturally we may ask if there exists a hemisystem of $\linear{q}$ invariant under $K$ for all $q$? For $q=13$ and $q=17$, we constructed the group $K$ and, as anticipated, found hemisystems stabilised by $K$, but to our surprise the sequence appears to stop there and for $q = 19$, $23$, $25$ and $27$ there are no hemisystems stabilised by $K$. (We were sufficiently surprised by this that we ran the linear program with a second integer programming package --- GLPK --- in addition to Gurobi.)

\subsection{Hemisystems invariant under $2^4.A_5$}
\label{sec:fixedgp}

A further interesting sequence is that for $\mathsf{H}(3,7^2)$ and $\mathsf{H}(3,11^2)$ there is a hemisystem with stabiliser of shape $2^4.A_5$. The stabiliser of the Segre hemisystem for $q=3$ also contains such a subgroup, and further calculations have verified the existence of a hemisystem invariant under $2^4.A_5$ when $q=19$. The group $\PGU(3,q)$ contains a subgroup $H$ isomorphic to $2^4.A_6$ for all $q\equiv 3\pmod 4$ (such a subgroup is usually referred to as a $\mathcal{C}_6$-group, or the normaliser of a symplectic type $r$-group, see for example \cite[\S 4.6]{KL}). The group $H$ contains two groups of shape $2^4.A_5$, corresponding to the two classes of $A_5$ in $A_6$. The group which arises as a stabiliser of a hemisystem for $q=3,7,11$ and $19$ is the one for which the $A_5$ acts transitively on the nontrivial elements of the $2^4$.  

\begin{problem}
Is there a hemisystem of $\linear{q}$ invariant under $2^4.A_5$ for all $q\equiv 3\pmod 4$?
\end{problem}

These hemisystems are especially intriguing (and also potentially harder to search for) as the order of their stabiliser is constant.

\section{BLT-sets}\label{qclans}

In this section, we list some of the known families of BLT-sets.
Suppose we are in the $3$-dimensional symplectic space $\mathsf{W}(3,q)$ 
defined by the form 
$\beta(\bvec{x},\bvec{y}) = x_1y_4-x_4y_1+x_2y_3-x_3y_2$. Then from Payne's \textit{$q$-clans} (see \cite{Payne85}) we can
construct BLT-sets of $\mathsf{W}(3,q)$. 
For the following, we note that a quadratic form $Q$ over $\GF(q)$ is \textit{anisotropic} if
$Q(\bvec{x})=0$ holds only when $\bvec{x}=\bvec{0}$. The following lemma
is straight-forward to prove (see \cite[p. 296]{BakerEbertPenttila}).

\begin{lemma}
\label{lem:qclanrep}
Consider the following lines $\mathcal{L}$ of $\mathsf{W}(3,q)$:
$$\ell_\infty:=
\begin{pmatrix}
0&0&1&0\\
0&0&0&1
\end{pmatrix},
\quad \ell_t:=
\begin{pmatrix}
1&0&f_t&t\\
0&1&g_t&f_t
\end{pmatrix}  \text{ for all }t\in\GF(q).$$
Then $\mathcal{L}$ is a BLT-set of lines of $\mathsf{W}(3,q)$ if and only if 
for all $t,u\in\GF(q)$, $t\ne u$, the following quadratic form on $\GF(q)^2\oplus \GF(q)^2$ is anisotropic:
$$(x,y)\mapsto (t-u)x^2+2(f_t-f_u)xy+(g_t-g_u)y^2.$$
\end{lemma}

We now summarise the maps $f$ and $g$ that generate the flock quadrangles used in this paper. Our information has been taken from \cite{MaskaThesis}. Four of the families are outlined in Table \ref{tab:qclans}.

\begin{table}[H]
\begin{center}
\begin{tabular}{p{3.2cm}|c|c|c|p{7cm}}
Flock quadrangle& Abbreviation & $f_t$&$g_t$ & Conditions\\
\hline
Linear & $\mathsf{H}(3,q^2)$ & 
$0$&$-nt$
&$n$ is a nonsquare in $\GF(q)$\\
Fisher-Thas-Walker-Kantor-Betten&$\ftwkb{q}$&
$\tfrac{3}{2}t^2$&$3t^3$
&$q\equiv 2\pmod{3}$\\
Kantor Monomial&$\mathsf{K}_2(q)$&
$\tfrac{5}{2}t^3$&$5t^5$
&
$q\equiv \pm 2\pmod{5}$, $5$ is a nonsquare in $\GF(q)$\\
Kantor-Knuth&$\mathsf{K}_1(q)$&
$0$&$-nt^\sigma$&
$n\in\GF(q)$ nonsquare, $q$ not prime, $1\ne \sigma\in\mathsf{Aut}(\GF(q))$\\
\hline
\end{tabular}
\end{center}
\caption{Functions $f$ and $g$ for some flock quadrangles. For each map, the variable $t$ runs over $\GF(q)$.}
\label{tab:qclans}
\end{table}

For the remaining flock quadrangles considered in this paper, the representation of the BLT-set as in Lemma \ref{lem:qclanrep} is more difficult to write down, so we resort to
a different model due to Penttila. Consider the dual \gq{} of $\mathsf{W}(3,q)$, the points and lines of the parabolic quadric
$\mathsf{Q}(4,q)$. So a BLT-set of lines of $\mathsf{W}(3,q)$ corresponds to a \textit{BLT-set of points} of $\mathsf{Q}(4,q)$. Consider $V:=\GF(q^2)\oplus \GF(q^2)\oplus\GF(q)$ 
as a vector space over $\GF(q)$, and define the following quadratic form on $V$:
$$(x,y,a)\mapsto x^{q+1}+y^{q+1}+a^2.$$
This quadratic form defines a parabolic quadric $\mathcal{Q}$ isomorphic to $\mathsf{Q}(4,q)$. 

The following models were taken from \cite{MaskaThesis}.

\subsubsection*{The Fisher BLT-sets:}

Fix an element $\beta\in\GF(q^2)$ with $\beta^{q+1}=-1$. Let 
$$\mathcal{P}=\{ (\beta x^2,0,1) \mid x^{q+1}=1\} \cup \{(0,\beta y^2,1)\mid y^{q+1}=1\}.$$
Then $\mathcal{P}$ defines a BLT-set of points of $\mathcal{Q}$.

\subsubsection*{The Penttila-Mondello BLT-sets:}

Suppose $q\equiv \pm 1\pmod{10}$ and fix $\beta,\gamma\in\GF(q^2)$ satisfying
$\beta^{q+1}=-\tfrac{4}{5}$ and $\gamma^{q+1}=-\tfrac{1}{5}$. Let 
$$\mathcal{P}=\{(\beta x^2,\gamma x^3,1)\mid x^{q+1}=1\}.$$
Then $\mathcal{P}$ is a BLT-set of points of $\mathcal{Q}$.
For $\pentmon{11}$, we may use a representation as in Lemma \ref{lem:qclanrep} given by the functions $f_t$ and $g_t$ in Table \ref{tab:PM11}.

\begin{table}[ht]
\begin{center}
\begin{tabular}{c|lllllllllll}
 $t$ & 0&1&2&3&4&5&6&7&8&9&10\\
\hline
$f_t$& 8 &0&7&4&8&0&1&5&0&0&0 \\
$g_t$ &1&8&3&2&5&6&10&9&4&7&0
\end{tabular}
\end{center}
\caption{The functions $f_t$ and $g_t$ for $\pentmon{11}$}
\label{tab:PM11}
\end{table}


\section{Computational methods}\label{section:methods}

The {\em point-line incidence matrix} of a \gq{} is the matrix $A$ with rows indexed by points and columns by lines such that
$$
A_{P,\ell} =
\begin{cases}
1, & P \text{ is on } \ell;\\ 0, & \text{otherwise}.
\end{cases}
$$
 In order to construct the point-line incidence matrix of a flock \gq, we used the GAP package \textsf{FinInG}\footnote{This can be found at \texttt{http://cage.ugent.be/geometry/fining.php}. This package is currently in development.}. This software can construct flock \gqs{} from the information given in 
 Section \ref{qclans}.
 
A hemisystem is a subset of the columns of $A$ that sum to $(s+1)/2\; \allones^T$ where $\allones$ is the all-ones (row) vector or, equivalently, a $\{0,1\}$-vector $\bvec{h}$ such that 
\begin{equation}
A\bvec{h}^T = (s+1)/2\; \allones^T.
\end{equation}

For all but the smallest \gqs, the matrix $A$ is so large that we cannot hope
to solve the equations completely. To reduce the problem, we assume the existence of some group $G$ stabilizing the hemisystem. Suppose that $G$ has orbits
$
\{ \mathcal{P}_1, \mathcal{P}_2, \ldots, \mathcal{P}_m \}
$
on points and 
$
\{ \mathcal{L}_1, \mathcal{L}_2, \ldots, \mathcal{L}_n \}
$
on lines. Then every point in a point-orbit $\mathcal{P}_i$ is incident with the same number of lines in the line-orbit $\mathcal{L}_j$. If we denote this number by $b_{ij}$ and define the $m \times n$ matrix $B = (b_{ij})$, then a $\{0,1\}$-vector $\bvec{h}$ such that 
\begin{equation}\label{tactical}
B\bvec{h}^T =  (s+1)/2\; \allones^T
\end{equation}
determines a hemisystem that is stabilised by the group $G$. 

There are a variety of approaches to solving equations such as \eqref{tactical}. In particular, the system of equations can be viewed either as an {\em integer linear program} or as a {\em constraint satisfaction problem}. After experimenting with software for each type of problem, we determined that the commercial integer programming package Gurobi \cite{gurobi} (available with a free academic license) was the most effective for our purposes.

A linear program attempts to find values for variables $x_1, x_2, \ldots, x_n$ that maximise (or minimise) a linear objective function subject to linear constraints. An {\em integer} linear program, or just integer program, is a linear program with the additional restriction that the variables must take integral values. Solving \eqref{tactical} does not involve any maximizing or minimizing and so the objective function can be taken to be a constant, say 0, and then any feasible solution $\bvec{x} = (x_1, x_2, \ldots, x_n)$ to the following integer program yields a hemisystem:
\[
	\begin{array}{lrcl}
	\textrm{Maximise:}   & 0 &      &   \\
	\textrm{subject to:} & B\bvec{x}^T    & =    & (s+1)/2\; \allones^T\\
			    & x_i     & \in & \{0, 1\}.
	\end{array}
\]

In order to find {\em all} the solutions to a given system of equations, the system is augmented as each solution is found with an additional constraint excluding that particular solution, and the system is then re-solved.  When all the solutions have been found and excluded, the resulting system has no integer feasible solutions. In order to exclude a particular solution $\bvec{h} = (h_1, h_2, \ldots, h_n)$ it suffices to add a constraint of the form
\[
\sum_{\{i \mid h_i = 1\}} x_i < \sum_i {h_i}
\]
which merely says that $\bvec{x}$ cannot agree with $\bvec{h}$ in every coordinate position, and so must differ in at least one place. In principle, a constraint of this form only eliminates vectors {\em identical} to $\bvec{h}$ and still permits the solver to investigate vectors that have almost all of their entries identical to $\bvec{h}$.  However, if we know an upper bound, say $\alpha$, on the size of the intersection of two hemisystems, then we can strengthen this constraint to 

\begin{equation}\label{constraint}
\sum_{\{i \mid h_i = 1\}} x_i \le \alpha
\end{equation}
without missing any hemisystems. The exhaustive search for hemisystems in $\linear{5}$ was made feasible by using two basic techniques to shorten the search time:
\begin{itemize}
\item Use the automorphism group of $\linear{5}$ to determine the largest possible
set of lines that can freely be assumed to be in a hemisystem.
\item Use knowledge of the possible intersection sizes of a hemisystem with the
two known hemisystems to add strong constraints of the same type as \eqref{constraint} during the search.
\end{itemize}

A more detailed description of the computation for $\linear{5}$ follows:

\smallskip

\textsc{Proof of Theorem \ref{thm:q=5} for $\linear{5}$.}

\smallskip

Let $G$ be the full automorphism group of $\linear{5}$ and let $\mathcal{H}$ be a hemisystem. As $G$ is transitive on the set of lines of $\linear{5}$ we can assume without loss of generality that $\ell_1\in\mathcal{H}$. Then the stabiliser $G_{\ell_1}$  has two orbits on the remaining lines, those disjoint from $\ell_1$ and those that meet $\ell_1$. It is easy to see that any hemisystem containing $\ell_1$ must contain a line disjoint from $\ell_1$ and so we can arbitrarily pick a second line, say $\ell_2$ , and assume without loss of generality that $\ell_1, \ell_2 \in {\mathcal H}$. This process can be continued in a semi-automated fashion as follows: suppose that we have a set $\ell_1,\ldots,\ell_i$ of lines that we can already assume are contained in $\mathcal{H}$, and consider the orbits of the setwise stabiliser $G_{\{\ell_1,\ldots,\ell_i\}}$ on lines. An orbit ${\mathcal O}$ is denoted {\em essential} if a search for a hemisystem that contains $\ell_1,\ldots,\ell_i$ but  does not contain {\em any} line from ${\mathcal O}$ is infeasible. If ${\mathcal O}$ is essential, then ${\mathcal H}$ contains at least one line from ${\mathcal O}$, and we can select $\ell_{i+1}$ arbitrarily from ${\mathcal O}$. This process can be continued until the set of lines is sufficiently large that its stabiliser is so small that it has no essential orbits. In this fashion, we found a particular set of 8 lines $\ell_1, \ldots, \ell_8$  that can be assumed to lie in $\mathcal{H}$.

The next important step was to determine that no hemisystem has a ``large'' intersection with either of the two known hemisystems. Let ${\mathcal H}_1$ and ${\mathcal H}_2$
be representatives of the two known hemisystems. First we found the maximum possible size in which {\em any} hemisystem (known or unknown) can intersect ${\mathcal H}_1$ by running the integer linear program where the objective function to be maximised is the sum of the variables corresponding to the lines in ${\mathcal H}_1$. This revealed that a hemisystem different from ${\mathcal H}_1$ can intersect ${\mathcal H}_1$ in at most 306 lines. By running the linear program again with the additional constraint that the intersection with ${\mathcal H}_1$ has size {\em exactly} 306, we determined all the hemisystems that intersect ${\mathcal H}_1$ in 306 lines and confirmed that no new hemisystems arose. We repeated this process with the ``next largest'' intersection, which proved to be size 300, then 282, then 270 and then 258, eventually confirming that any hemisystem that meets ${\mathcal H}_1$ in 258 or more lines is isomorphic to either ${\mathcal H}_1$ or ${\mathcal H}_2$.  Similar results were obtained for ${\mathcal H}_2$ and similarly we determined that any hemisystem meeting ${\mathcal H}_2$ in 258 or more lines is isomorphic to ${\mathcal H}_1$ or ${\mathcal H}_2$.

Finally, the exhaustive search is run where the variables corresponding to $\ell_1, \ldots, \ell_8$ are initially set to 1 and every time a hemisystem is found, it is excluded by adding a constraint similar to \eqref{constraint} with $\alpha = 257$. Notice that this constraint is much stronger than simply excluding the hemisystem that has just been found and will exclude other hemisystems. However if the just-found hemisystem is one of the two known ones, then the ``extra'' hemisystems that are excluded by the constraint are necessarily isomorphic to the known ones, and hence not of interest. Therefore if unknown hemisystems do exist, then at least one of them will be discovered by the search.  As this does not occur, we conclude that there are no other hemisystems of $\linear{5}$. \qed

In this computation, there is a trade-off involved in choosing the value 258 used in the constraints to exclude solutions as they are found. Using a lower value would make the final exhaustive part of the search run faster, but it would take longer to establish that only known hemisystems intersect ${\mathcal H}_1$ or ${\mathcal H}_2$ in that many lines. 

The computation for $\ftwkb{5}$ was done in an exactly analogous fashion.

\section{A summary of the known hemisystems of flock quadrangles}\label{known}

In this section we catalogue all the known hemisystems of lines of flock quadrangles of order $(s^2,s)$ for $s\le 11$. These include
those which arise in the pre-existing literature, those obtained via Theorem \ref{construction}, and numerous further examples constructed by computer.  Each row of the table describes a complementary pair of hemisystems; the column SC (for ``self-complementary'') indicates whether the hemisystem is equivalent to its complement in which case it contributes just 1 to the total count of hemisystems.

The tables contain an exhaustive listing of all the hemisystems that arise by Theorem \ref{construction} and are complete for the known \gqs{} of order up to $(5^2,5)$. However there may be many more hemisystems, though necessarily with small automorphism groups, that remain to be found.

\begin{proposition}
Let $\mathcal{H}$ be a hemisystem of a flock quadrangle of order $(s^2,s)$ with $s\le 9$ such that $\mathcal{H}$ arises from Theorem \ref{construction}. Then $\mathcal{H}$ appears in one of the tables in this section.
\end{proposition}

We also list all hemisystems arising from Theorem \ref{construction} for $\linear{11}$ in Table \ref{tab:H311}. Due to the large number of hemisystems of Type I for the remaining flock quadrangles of order $(11^2,11)$, they are listed in the Appendix, which is only included in the version of this paper on the \textsc{arxiv}.

\begin{quote}
\begin{framed}
\begin{center}
\emph{Reconstruction of the hemisystems from the data}
\end{center}
The data given for the Type I hemisystems in our tables is sufficient to reconstruct the actual hemisystem given some additional knowledge about the particular choices that have been made for the variables in the construction. First of all, the finite fields in \textsf{GAP} have a determined primitive element and the ordering of the elements of the field is first graded by towers of subfields, and then by exponents of the primitive element. Matrices in \textsf{GAP} are ordered lexicographically, row by row.
The point $P$ is $(1,0,0,0,0,0)$ and the BLT-sets are the ones given in Section \ref{qclans}.  Each totally isotropic plane can be represented uniquely by a $3\times 6$ matrix written in Hermite normal form, whose row space gives us the corresponding 3-dimensional vector subspace. The totally isotropic planes on $\ell$ are sorted by sorting the corresponding $3\times 6$ matrices into lexicographic order, and indexed by $\{1,\ldots,q+1\}$. The chosen subset ${\mathcal S}$ is given by a $(q-1)/2$ subset of this index set.
\end{framed}
\end{quote}

\subsection{Linear, $\linear{3}$}

Segre \cite{Segre65} established that there is just one example of a
hemisystem (up to projectivity) in $\linear{3}$. The strongly regular graph (and partial quadrangle) arising is the Gewirtz graph on $56$ vertices.

\begin{table}[H]
\begin{center}
\begin{tabular}{l|l|c|c|c|l}
Group & \multicolumn{1}{|c|}{Size}& SC &  Construction/Author(s) & $\ell$ & Subset ${\mathcal S}$\\   
\hline   	
$\mathsf{PSL}(3,4).2$ & 40320 & true & Theorem \ref{construction}, Segre \cite{Segre65},&
$\left[\begin{smallmatrix}
0&1&0&1&0&0\\  0&0& 1&1 &1& 0\\
 \end{smallmatrix}\right]$ 
& any  \\ 
&&&Sections \ref{sec:metacyclic}, \ref{sec:triangular} and \ref{sec:fixedgp}&& \\ \hline
\end{tabular}
\end{center}
\caption{The hemisystem of $\linear{3}$.}
\label{tab:H39}
\end{table}

\subsection{Linear, $\linear{5}$}

The full automorphism group of this \gq{} is $\PGammaU(4,5)$ which has order
$2^9 \times 3^4 \times 5^6 \times 7 \times 13$. There were two previously known
hemisystems in this \gq{} and our computer searches have confirmed that there are no more.

\begin{table}[H]
\begin{center}
\begin{tabular}{l|r|c|c|c|l}
Group & \multicolumn{1}{|c|}{Size}& SC &  Construction/Author(s) & $\ell$ & Subset ${\mathcal S}$\\   
\hline
$\mathsf{P\Sigma L}(2,25)$&15600& true &  Theorem \ref{construction}, Cossidente--Penttila \cite{CossidentePenttila05}&
$\left[\begin{smallmatrix}
0&1&0&1&0&0\\  0&0& 1&0 &1& 0\\
 \end{smallmatrix}\right]$ &
any \\
\hline
$(3\cdot A_7).2$&15120& true &  Cossidente--Penttila \cite{CossidentePenttila05}, Sections \ref{sec:metacyclic} and \ref{sec:triangular}\\
\cline{1-4}
\end{tabular}
\end{center}
\caption{The hemisystems of $\linear{5}$.}
\label{tab:H35}
\end{table}

\subsection{Fisher-Thas/Walker/Kantor/Betten, $\ftwkb{5}$}\label{FTWKB5}

The full automorphism group of this \gq{} is $5^{1+4}:(\mathsf{SL}(2,9):C_4)$, which has order $2^6\times 3^2\times 5^6$.  There was one previously known hemisystem of this \gq{} in the literature and Theorem \ref{construction} yields a second example. Our computer searches uncovered a third example with group $S_3$, and confirmed that there are no more.

\begin{table}[H]
\begin{center}
\begin{tabular}{l|r|c|c|c|l}
Group & \multicolumn{1}{|c|}{Size}& SC &  Construction/Author(s) & $\ell$ & Subset ${\mathcal S}$\\   
\hline
$C_5^2:(C_4\times S_3)$& 600 &false&Theorem \ref{construction}&
$\left[\begin{smallmatrix}
0&1&0&0&1&0\\  0&0& 1&1 &0 & 0\\
 \end{smallmatrix}\right]$ &
any\\
\hline
$\AGL(1,5)\times S_3$&120& false& Bamberg--De Clerck--Durante \cite{BambergDeClerckDurante09}\\
$S_3$& 6& false& \textbf{New}\\
\cline{1-4}
\end{tabular}
\end{center}
\caption{The hemisystems of $\ftwkb{5}$.}
\label{tab:ftwkb5}
\end{table}

\subsection{Linear, $\linear{7}$}

The full automorphism group of this \gq{} is $\PGammaU(4,7)$, which has order $2^{13}\times 3^2\times 5^2\times 7^6\times 43$. There were five previously known hemisystems in this quadrangle and our computer searches have uncovered a sixth.

\begin{table}[H]
\begin{center}
\begin{tabular}{l|r|c|c|c|l}
Group & \multicolumn{1}{|c|}{Size}& SC &  Construction/Author(s) & $\ell$ & Subset ${\mathcal S}$\\   
\hline
$\mathsf{P\Sigma L}(2,49)$&117600 &true&Theorem \ref{construction}, Cossidente--Penttila \cite{CossidentePenttila05}&
$\left[\begin{smallmatrix}
0&1&0&1&0&0\\  0&0& 1&1 &1& 0\\
 \end{smallmatrix}\right]$& 
 $\{1,3,4\} $ \\
$C_2 \times (C_7^2 : Q_{16})$&1568&true&Penttila (personal communication), &
& 
 $\{1,3,5\} $\\
&&&Theorem \ref{construction}&&\\
\hline
$2^4.A_5$&960&true&Bamberg--Kelly--Law--Penttila  \cite{BKLP07}\\
&&&Section \ref{sec:fixedgp}\\
$C_2\times (C_{43}:C_{6})$&516&true&Bamberg--Kelly--Law--Penttila  \cite{BKLP07}\\
$C_8^2:S_3$ & 384&true& \textbf{New}, Section \ref{sec:triangular}\\
$C_2\times\mathsf{PSL}(2,7)$&336&true&Cossidente--Penttila \cite{CossidentePenttila09}, Section \ref{sec:metacyclic}  \\
\cline{1-4}
\end{tabular}
\end{center}
\caption{Known hemisystems of $\linear{7}$.}
\end{table}
  
\subsection{Kantor Monomial, $\kmonom{7}$}

The full automorphism group of this \gq{} is 
$7^{1+4}:(C_3 \times (Q_8:(\mathsf{SL}(2,3).2):2))$, which has order $2^8\times 3^2\times 7^5$.  In addition to the 14 examples obtained by Theorem \ref{construction}, we have found a further 15 hemisystems; all are listed in Table~\ref{tab:km7}.

\begin{table}[H]
\begin{center}
\begin{tabular}{l|r|c|c|c|l}
Group & \multicolumn{1}{|c|}{Size}& SC &  Construction/Author(s) & $\ell$ & Subset ${\mathcal S}$\\   
\hline
$C_7^2:(C_3\times \mathsf{SL}(2,3))$&3528&false&Theorem \ref{construction}&
$\left[\begin{smallmatrix}
0&1&0&1&0&0\\  0& 0&1 &0& 1&0\\
 \end{smallmatrix}\right]$&  
 $\{1,3,4\} $\\
$C_7^2:(\mathsf{SL}(2,3).2)$&2352&false&Theorem \ref{construction}&
 &  
 $\{1,3,5\} $\\
 \hline
$C_7^2:(Q_{16}\times C_3)$&2352&false&Theorem \ref{construction}&
$\left[\begin{smallmatrix}
0&1&0&0&1&0\\  0&0& 1&4 &0& 0\\
 \end{smallmatrix}\right]$ &  
 $\{1,3,4\} $\\
$(C_7^2:Q_{16})\times C_2$&1568&false&Theorem \ref{construction}&
 &  
 $\{1,3,5\} $ \\
  \hline
$C_7^2:(C_6\times C_3)$&882&true&Theorem \ref{construction}&
$\left[\begin{smallmatrix}
0&1&0&0&1&0\\ 0&0&1&3&0&0\\
 \end{smallmatrix}\right]$  &  
 $\{1, 3, 4\} $\\
$C_7^2:(C_3:C_4)$&588&true&Theorem \ref{construction}&
 &  
 $\{1, 3, 5\} $ \\
  \hline
$C_7^2:C_{12}$&588&false&Theorem \ref{construction}&
$\left[\begin{smallmatrix}
0&1&0&0&1&0\\ 0&0&1&1&0&0\\
 \end{smallmatrix}\right]$  &  
 $\{1,3,4\} $\\
$C_7^2:Q_8$&392&false&Theorem \ref{construction}&
 &  
 $\{1, 3, 5\} $\\
\hline
$C_3\times F_{42}$&126&false&\textbf{New}\\
$C_3\times F_{42}$&126&true&\textbf{New}\\
$C_2\times (C_7:C_3)$&42&false&\textbf{New}\\
$\AGL(1,7)$&42&false&\textbf{New}\\
$(C_2\times Q_8):C_2$&32&false&\textbf{New}\\
$(C_2\times Q_8):C_2$&32&false&\textbf{New}\\
$C_7:C_3$&21&false&\textbf{New}\\
$C_7:C_3$&21&true&\textbf{New}\\
$C_3$ &3 &true&\textbf{New}\\
\cline{1-4}
\end{tabular}
\end{center}
\caption{Known hemisystems of $\kmonom{7}$.}
\label{tab:km7}
\end{table}

%
%

\subsection{Linear, $\mathsf{H}(3,9^2)$}

The full automorphism group of this \gq{} is $\PGammaU(4,9)$, which has order $2^{12}\times3^{12}\times 5^3\times41\times73$. In addition to the two previously known hemisystems, we found two more arising from Theorem~\ref{construction} and three others; all are listed in Table~\ref{tab:lin9}.

\begin{table}[H]
\begin{center}
\begin{tabular}{l|r|c|c|c|l}
Group & \multicolumn{1}{|c|}{Size}& SC &  Construction/Author(s) & $\ell$ & Subset ${\mathcal S}$\\   
\hline
$\mathsf{P\Sigma L}(2,81)$& 1062720 &true&Theorem \ref{construction}, Cossidente--Penttila \cite{CossidentePenttila05}&
$\left[\begin{smallmatrix}
0&1&0&1&0&0\\  0&0& 1&0 &1& 0\\
 \end{smallmatrix}\right]$  &  
 $\{1,3,4,5\}$\\
$C_3^4:(C_{20}:C_4)$  & 6480 & true &Theorem \ref{construction}&
 &  
 $\{1,3,5,6\}$\\
$C_3^4:(C_{5}:C_8)$  & 3240& true & Theorem \ref{construction}&
 &  
 $\{1,3,5,9\}$ \\
\hline
 $C_{73}:C_{12}$ & 876 & true &Bamberg--Kelly--Law--Penttila  \cite{BKLP07},\\
&&& Section \ref{sec:metacyclic} \\
$(C_{10}^2 : C_4) : C_3$ &1200& true &\textbf{New}, Section \ref{sec:triangular}\\
$C_{10}^2:S_3$ &600& true &\textbf{New}, Section \ref{sec:triangular}\\
   $(C_5 \times (C_5 : C_4)) : C_4$  &400&true &\textbf{New}\\
\cline{1-4}
\end{tabular}
\end{center}
\caption{Known hemisystems of $\linear{9}$.}
\label{tab:lin9}
\end{table}

\subsection{Kantor-Knuth, $\mathsf{K}_1(9)$}

The full automorphism group of this \gq{} is $E_9 : (((\SL(2,9).C_4) : C_8) : C_2)$ where $E_9$ is the Heisenberg group of order $9^5$ with centre of order $9$. The order of the automorphism group is $2^{10}\times 3^{12}\times5$.

\begin{table}[H]
\begin{center}
\begin{tabular}{l|r|c|c|c|l}
Group & \multicolumn{1}{|c|}{Size}& SC &  Construction/Author(s) & $\ell$ & Subset ${\mathcal S}$\\   
\hline
$C_3^4:(C_4\times C_{8})$&2592&true&Theorem \ref{construction}&
$\left[\begin{smallmatrix}
0&1&0&1&0&0\\ 0&0&1&0&1&0\\
 \end{smallmatrix}\right]$   &  
 $\{1,3,5,9\}$ \\
$C_3^4:(C_2\times C_{8})$&1296&true&Theorem \ref{construction}&
 &  
 $\{1,3,5,6\}$ \\
$C_3^4:C_{8}$&648&true&Theorem \ref{construction}&
 &  
 $\{1,3,4,5\}$ \\
\hline
$C_4 \times \AGL(1,9)$   &288&true &\textbf{New}\\
$\AGL(1,9)$ & 72&true &\textbf{New}\\
\cline{1-4}
\end{tabular}
\end{center}
\caption{Known hemisystems of $\kknuth{9}$.}
\end{table}

\subsection{Fisher, $\fish{9}$}

 The full automorphism group of this \gq{} is 
$E_9 :(C_5^2: (D_{16}.Q_8))$ which has order $2^{7} \times 3^{10} \times 5^{2}$. (Here
$E_9$ is the Heisenberg group of order $9^5$ with centre of order $9$.)

\begin{center}
\begin{table}[H]
\begin{tabular}{l|r|c|c|c|l}
Group & \multicolumn{1}{|c|}{Size}& SC &  Construction/Author(s) & $\ell$ & Subset ${\mathcal S}$\\   
\hline
$C_3^4 : C_5 : C_8$ & 3240 & false  &Theorem \ref{construction} & $\left[\begin{smallmatrix} 0 & 1 & 0 & 0 & 0 & 0 &  \\ 0 & 0 & 1 & 0 & 0 & 0 &  \\  \end{smallmatrix}\right]$ &\{1, 3, 4, 5\}\\ 
$C_3^4 : C_{20} : C_4$ & 6480 & false  &Theorem \ref{construction} &  &\{1, 3, 5, 6\}\\ 
$C_3^4 : C_5 : (C_4 . (C_4 \times C_2))$ & 
12960 & false  &Theorem \ref{construction} &  &\{1, 3, 5, 9\}\\ 
\hline
$C_3^4 : C_2$ & 162 & false  &Theorem \ref{construction} & $\left[\begin{smallmatrix} 0 & 1 & 0 & 0 & 0 & 0 &  \\ 0 & 0 & 1 & z & 0 & 0 &  \\  \end{smallmatrix}\right]$ &\{1, 3, 4, 5\}\\ 
$C_3^4 : C_4$ & 324 & false  &Theorem \ref{construction} &  &\{1, 3, 5, 6\}\\ 
$C_3^4 : C_8$ & 648 & false  &Theorem \ref{construction} &  &\{1, 3, 5, 9\}\\ 
\hline
$C_3^4 : C_4$ & 324 & true  &Theorem \ref{construction} & $\left[\begin{smallmatrix} 0 & 1 & 0 & 1 & 1 & 0 &  \\ 0 & 0 & 1 & 2 & 1 & 0 &  \\  \end{smallmatrix}\right]$ &\{1, 3, 4, 5\}\\ 
$C_3^4 : C_4 \times C_2$ & 648 & true  &Theorem \ref{construction} &  &\{1, 3, 5, 6\}\\ 
$C_3^4 : C_8 \times C_2$ & 1296 & true  &Theorem \ref{construction} &  &\{1, 3, 5, 9\}\\ 
\hline
$C_3^4 : C_4$ & 324 & true  &Theorem \ref{construction} & $\left[\begin{smallmatrix} 0 & 1 & 0 & 1 & 1 & 0 &  \\ 0 & 0 & 1 & z^2 & 1 & 0 &  \\  \end{smallmatrix}\right]$ &\{1, 3, 4, 5\}\\ 
$C_3^4 : C_4 \times C_2$ & 648 & true  &Theorem \ref{construction} &  &\{1, 3, 5, 6\}\\ 
$C_3^4 : C_8 \times C_2$ & 1296 & true  &Theorem \ref{construction} &  &\{1, 3, 5, 9\}\\ 
\hline
$C_3^4 : C_4$ & 324 & true  &Theorem \ref{construction} & $\left[\begin{smallmatrix} 0 & 1 & 0 & 1 & z^2 & 0 &  \\ 0 & 0 & 1 & z^3 & 1 & 0 &  \\  \end{smallmatrix}\right]$ &\{1, 3, 4, 5\}\\ 
$C_3^4 : C_4 \times C_2$ & 648 & true  &Theorem \ref{construction} &  &\{1, 3, 5, 6\}\\ 
$C_3^4 : C_8 \times C_2$ & 1296 & true  &Theorem \ref{construction} &  &\{1, 3, 5, 9\}\\ 

\hline 
 $ C_2 \times \AGL(1,9)$   &144& true & \textbf{New}\\
 $\AGL(1,9)$ & 72 & $4 \times 2 + 4$ & \textbf{New}\\
\cline{1-4}
\end{tabular}
\caption{Known hemisystems of $\fish{9}$.}
\end{table}
\end{center}

\subsection{Linear, $\linear{11}$}
The full automorphism group of this \gq{} is $\PGammaU(4,11)$, which has order $2^{10} \times 3^{4} \times 5^{2} \times 11^{6} \times 37 \times 61$.

\begin{center}
\begin{table}[H]
\begin{tabular}{l|r|c|c|c|l}
Group & \multicolumn{1}{|c|}{Size}& SC &  Construction/Author(s) & $\ell$ & Subset ${\mathcal S}$\\   
\hline
$\mathsf{P\Sigma L}(2,121)$& 1771440& true & Theorem \ref{construction} &$\left[\begin{smallmatrix} 0 & 1 & 0 & 1 & 0 & 0 &  \\ 0 & 0 & 1 & 2 & 1 & 0 &  \\  \end{smallmatrix}\right]$ & \{ 1, 3, 4, 5, 8 \}\\ 
$C_{11}^2 : C_2 \times (C_3 : Q_8)$ & 5808 & true & Theorem \ref{construction} & & \{ 1, 3, 4, 5, 7 \}\\ 
$C_{11}^2 : C_2 \times (C_3 : Q_8)$ & 5808 & true & Theorem \ref{construction} & & \{ 1, 3, 4, 5, 9 \}\\ 
$C_{11}^2 : C_3 : Q_8$ & 2904 & true & Theorem \ref{construction} & & \{ 1, 3, 4, 5, 10 \}\\ 
$C_{11}^2 : C_3 : Q_8$ & 2904 & true & Theorem \ref{construction} & & \{ 1, 3, 4, 5, 11 \}\\ 
$C_{11}^2 : C_3 : Q_8$ & 2904 & true & Theorem \ref{construction} & & \{ 1, 3, 4, 5, 6 \}\\ 

\hline
$3.A_6.2$  &2160 & true    &\textbf{New}\\
$C_{333}: C_6$ & 1998& true &\textbf{New}, Section \ref{sec:metacyclic}\\
$2^4.A_5$ & 960 & true&\textbf{New}, Section \ref{sec:fixedgp}\\
$C_{12}^2:S_3$& 864 & true &\textbf{New}, Section \ref{sec:triangular}\\
 $C_{12}^2:S_3$& 864 & true &\textbf{New}, Section \ref{sec:triangular}\\
$C_{111}: C_6$ & 666& false & \textbf{New}, Section \ref{sec:metacyclic}\\
\cline{1-4}
\end{tabular}
\caption{Known hemisystems of $\linear{11}$.}
\label{tab:H311}
\end{table}
\end{center}

\subsection{Fisher-Thas-Walker-Kantor-Betten, $\ftwkb{11}$}

The full automorphism group of this \gq{} is $11^{1+4}\rtimes \mathrm{GL}(2,11)$ which has order $2^{4} \times 3 \times 5^{2} \times 11^{6}$. There are 20 hemisystems of Type I, listed in the Appendix\footnote{Due to its size, this Appendix is only included in the \textsc{arxiv} version of this paper}, and we do not know any other hemisystems in this \gq{}.

\subsection{Fisher, $\fish{11}$}

The full automorphism group of this \gq{} is $11^{1+4}:(C_5 \times (((C_3 \times (C_3 : C_4)) : Q_8) : C_2))$ which has order $2^{6} \times 3^{2} \times 5 \times 11^{5}$. There are 90 hemisystems of Type I, listed in the Appendix\footnotemark[3], and we know 12 further hemisystems listed in Table~\ref{tab:nonbgrfi11}.

{\small
\begin{center}
\begin{longtable}{|l|r|c|}
\hline
\multicolumn{1}{|c|}{Group} & \multicolumn{1}{c|}{Size}& \multicolumn{1}{c|}{Number}\\
\hline
$\AGL(1,11)$ &110& $6 \times 2$ \\
\hline
\caption{Non Type I hemisystems of $\fish{11}$}
\label{tab:nonbgrfi11}
\end{longtable}
\end{center}
}

\subsection{Penttila-Mondello, $\pentmon{11}$}

The full automorphism of this \gq{} is
$11^{1+4}\rtimes (C_5\times (C_3 \times \mathrm{SL}(2,3).2):2 )$ which has order $2^{5} \times 3^{2} \times 5 \times 11^{5}$.  There are 164 hemisystems of Type I, listed in the Appendix\footnotemark[3], and we know 36 further hemisystems listed in Table~\ref{tab:nonbgrmon11}.

{\small
\begin{center}
\begin{longtable}{|l|r|c|}
\hline
\multicolumn{1}{|c|}{Group} & \multicolumn{1}{c|}{Size}& \multicolumn{1}{c|}{Number}\\
\hline
$\AGL(1,11)$ &110& $18 \times 2$ \\
\hline
\caption{Non Type I hemisystems of $\pentmon{11}$}
\label{tab:nonbgrmon11}
\end{longtable}
\end{center}
}

%
%

\section{Open Problems}
\label{sec:probs}

We saw in Section \ref{section:construction} that in any infinite family of \gqs{} of order $(q^2,q)$ the number of hemisystems arising from Theorem~\ref{construction} grows exponentially in $q$. Hemisystems that do not arise from Theorem \ref{construction} are then of particular interest.

\begin{problem}
Does every flock \gq{} of order $(s^2,s)$ with $s \ge 7$ contain a hemisystem that does not arise from Theorem~\ref{construction}?
\end{problem}
We have found such hemisystems in all of the \gqs{} that we have examined with the exception of  the small cases ($\linear{3}$ and $\linear{5}$) and $\ftwkb{11}$.

Although we have outlined two possibilities for infinite families of hemisystems in $\linear{q}$ in Section \ref{sec:inffamilies}, we do not have any proven general constructions for hemisystems other than Theorem~\ref{construction}.

\begin{problem}
Find a natural construction for an infinite family of hemisystems (not of Type I) in $\linear{q}$ or in one of the known families of non-classical \gqs{}.
\end{problem}

By Theorem \ref{thm:stab}, a hemisystem coming from Theorem \ref{construction} is invariant under a particular elementary abelian group of order $q^2$ denoted by $T$.  However, we do not know if the converse is true.

\begin{problem}
Are there hemisystems invariant under the elementary abelian group $T$ of order $q^2$ described in Theorem~\ref{thm:stab} that do not arise from Theorem \ref{construction}? 
\end{problem}

At the other end of the symmetry spectrum, we currently do not know of any hemisystems with a trivial group. However this is not surprising, as almost all of our searches have assumed the existence of symmetries. 
\begin{problem}
 Is there a hemisystem with trivial group? 
\end{problem}
We expect a positive answer, although it may be challenging to find such a hemisystem.

Each hemisystem gives a strongly regular graph and the stabiliser of the hemisystem in the automorphism group of the \gq{} gives a group of automorphisms of the strongly regular graph. In all cases investigated so far, the automorphism group of the strongly regular graph is induced by the stabiliser of the hemisystem in the automorphism group of the \gq. It is not apparent why this should always be the case.
\begin{problem}
Is the full automorphism group of the strongly regular graph obtained from a hemisystem always induced by the stabiliser of the hemisystem in the automorphism group of the \gq?
\end{problem}

\begin{problem}
Are there hemisystems in different \gqs{} whose associated strongly regular graphs are isomorphic?
\end{problem}

\section*{Acknowledgements} 

The authors are extremely grateful to Simon Guest for his computational assistance.

\bibliographystyle{amsplain}
\providecommand{\bysame}{\leavevmode\hbox to3em{\hrulefill}\thinspace}
\providecommand{\MR}{\relax\ifhmode\unskip\space\fi MR }
\providecommand{\MRhref}[2]{%
  \href{http://www.ams.org/mathscinet-getitem?mr=#1}{#2}
}
\providecommand{\href}[2]{#2}

\appendix
\section{Tables of Type I hemisystems}\label{tables}

\subsection{Type I hemisystems of $\ftwkb{11}$}

{\small
\begin{center}
\begin{longtable}{l|r|c|c|l}
Group & \multicolumn{1}{|c|}{Size}& SC &  $\ell$ & Subset ${\mathcal S}$\\   
\hline
$C_{11}^2 : C_4$ & 484 & false  & $\left[\begin{smallmatrix} 0 & 1 & 0 & 0 & 1 & 0 &  \\ 0 & 0 & 1 & 2 & 0 & 0 &  \\  \end{smallmatrix}\right]$ &\{1, 3, 4, 5, 6\}\\ 
$C_{11}^2 : C_4$ & 484 & false  &  &\{1, 3, 4, 5, 7\}\\ 
$C_{11}^2 : C_{20}$ & 2420 & false  &  &\{1, 3, 4, 5, 8\}\\ 
$C_{11}^2 : C_4$ & 484 & false  &  &\{1, 3, 4, 5, 9\}\\ 
$C_{11}^2 : C_4$ & 484 & false  &  &\{1, 3, 4, 5, 10\}\\ 
$C_{11}^2 : C_4$ & 484 & false  &  &\{1, 3, 4, 5, 11\}\\ 
$C_{11}^2 : C_4$ & 484 & false  &  &\{1, 3, 4, 5, 12\}\\ 
$C_{11}^2 : C_4$ & 484 & false  &  &\{1, 3, 4, 6, 10\}\\ 
$C_{11}^2 : C_4$ & 484 & false  &  &\{1, 3, 4, 7, 10\}\\ 
$C_{11}^2 : C_{20}$ & 2420 & false  &  &\{1, 3, 4, 7, 12\}\\ 

\hline
\caption{Type I hemisystems of $\ftwkb{11}$.}
\label{tab:Type1FTWKB11}
\end{longtable}
\end{center}
}

\subsection{Type I hemisystems of $\fish{11}$}

{\small
\begin{center}
\begin{longtable}{l|r|c|c|l}
Group & \multicolumn{1}{|c|}{Size}& SC &  $\ell$ & Subset ${\mathcal S}$\\   
\hline
$C_{11}^2 : C_2$ & 242 & false    & $\left[\begin{smallmatrix} 0 & 1 & 0 & 0 & 1 & 0 &  \\ 0 & 0 & 1 & 1 & 0 & 0 &  \\  \end{smallmatrix}\right]$ & \{ 1, 3, 4, 5, 6 \}\\ 
$C_{11}^2 : C_4$ & 484 & false    &  & \{ 1, 3, 4, 5, 7 \}\\ 
$C_{11}^2 : C_{10}$ & 1210 & false    &  & \{ 1, 3, 4, 5, 8 \}\\ 
$C_{11}^2 : C_4$ & 484 & false    &  & \{ 1, 3, 4, 5, 9 \}\\ 
$C_{11}^2 : C_2$ & 242 & false    &  & \{ 1, 3, 4, 5, 10 \}\\ 
$C_{11}^2 : C_2$ & 242 & false    &  & \{ 1, 3, 4, 5, 11 \}\\ 
\hline
$C_{11}^2 : C_2$ & 242 & false    & $\left[\begin{smallmatrix} 0 & 1 & 0 & 0 & 1 & 0 &  \\ 0 & 0 & 1 & 2 & 0 & 0 &  \\  \end{smallmatrix}\right]$ & \{ 1, 3, 4, 5, 6 \}\\ 
$C_{11}^2 : C_4$ & 484 & false    &  & \{ 1, 3, 4, 5, 7 \}\\ 
$C_{11}^2 : C_{10}$ & 1210 & false    &  & \{ 1, 3, 4, 5, 8 \}\\ 
$C_{11}^2 : C_4$ & 484 & false    &  & \{ 1, 3, 4, 5, 9 \}\\ 
$C_{11}^2 : C_2$ & 242 & false    &  & \{ 1, 3, 4, 5, 10 \}\\ 
$C_{11}^2 : C_2$ & 242 & false    &  & \{ 1, 3, 4, 5, 11 \}\\ 
\hline
$C_{11}^2 : C_2$ & 242 & true    & $\left[\begin{smallmatrix} 0 & 1 & 0 & 0 & 1 & 0 &  \\ 0 & 0 & 1 & 10 & 0 & 0 &  \\  \end{smallmatrix}\right]$ & \{ 1, 3, 4, 5, 6 \}\\ 
$C_{11}^2 : C_4$ & 484 & true    &  & \{ 1, 3, 4, 5, 7 \}\\ 
$C_{11}^2 : C_{10}$ & 1210 & true    &  & \{ 1, 3, 4, 5, 8 \}\\ 
$C_{11}^2 : C_4$ & 484 & true    &  & \{ 1, 3, 4, 5, 9 \}\\ 
$C_{11}^2 : C_2$ & 242 & true    &  & \{ 1, 3, 4, 5, 10 \}\\ 
$C_{11}^2 : C_2$ & 242 & true    &  & \{ 1, 3, 4, 5, 11 \}\\ 
\hline
$C_{11}^2 : C_3 : Q_8$ & 2904 & false    & $\left[\begin{smallmatrix} 0 & 1 & 0 & 0 & 1 & 0 &  \\ 0 & 0 & 1 & 9 & 0 & 0 &  \\  \end{smallmatrix}\right]$ & \{ 1, 3, 4, 5, 6 \}\\ 
$C_{11}^2 : C_2 \times (C_3 : Q_8)$ & 5808 & false    &  & \{ 1, 3, 4, 5, 7 \}\\ 
$C_{11}^2 : C_5 \times (C_3 : Q_8)$ & 14520 & false    &  & \{ 1, 3, 4, 5, 8 \}\\ 
$C_{11}^2 : C_2 \times (C_3 : Q_8)$ & 5808 & false    &  & \{ 1, 3, 4, 5, 9 \}\\ 
$C_{11}^2 : C_3 : Q_8$ & 2904 & false    &  & \{ 1, 3, 4, 5, 10 \}\\ 
$C_{11}^2 : C_3 : Q_8$ & 2904 & false    &  & \{ 1, 3, 4, 5, 11 \}\\ 
\hline
$C_{11}^2 : C_4$ & 484 & false    & $\left[\begin{smallmatrix} 0 & 1 & 0 & 1 & 2 & 0 &  \\ 0 & 0 & 1 & 1 & 1 & 0 &  \\  \end{smallmatrix}\right]$ & \{ 1, 3, 4, 5, 6 \}\\ 
$C_{11}^2 : Q_8$ & 968 & false    &  & \{ 1, 3, 4, 5, 7 \}\\ 
$C_{11}^2 : C_{20}$ & 2420 & false    &  & \{ 1, 3, 4, 5, 8 \}\\ 
$C_{11}^2 : Q_8$ & 968 & false    &  & \{ 1, 3, 4, 5, 9 \}\\ 
$C_{11}^2 : C_4$ & 484 & false    &  & \{ 1, 3, 4, 5, 10 \}\\ 
$C_{11}^2 : C_4$ & 484 & false    &  & \{ 1, 3, 4, 5, 11 \}\\ 
\hline
$C_{11}^2 : C_4$ & 484 & false    & $\left[\begin{smallmatrix} 0 & 1 & 0 & 1 & 2 & 0 &  \\ 0 & 0 & 1 & 5 & 1 & 0 &  \\  \end{smallmatrix}\right]$ & \{ 1, 3, 4, 5, 6 \}\\ 
$C_{11}^2 : Q_8$ & 968 & false    &  & \{ 1, 3, 4, 5, 7 \}\\ 
$C_{11}^2 : C_{20}$ & 2420 & false    &  & \{ 1, 3, 4, 5, 8 \}\\ 
$C_{11}^2 : Q_8$ & 968 & false    &  & \{ 1, 3, 4, 5, 9 \}\\ 
$C_{11}^2 : C_4$ & 484 & false    &  & \{ 1, 3, 4, 5, 10 \}\\ 
$C_{11}^2 : C_4$ & 484 & false    &  & \{ 1, 3, 4, 5, 11 \}\\ 
\hline
$C_{11}^2 : C_4$ & 484 & false    & $\left[\begin{smallmatrix} 0 & 1 & 0 & 1 & 4 & 0 &  \\ 0 & 0 & 1 & 8 & 1 & 0 &  \\  \end{smallmatrix}\right]$ & \{ 1, 3, 4, 5, 6 \}\\ 
$C_{11}^2 : Q_8$ & 968 & false    &  & \{ 1, 3, 4, 5, 7 \}\\ 
$C_{11}^2 : C_{20}$ & 2420 & false    &  & \{ 1, 3, 4, 5, 8 \}\\ 
$C_{11}^2 : Q_8$ & 968 & false    &  & \{ 1, 3, 4, 5, 9 \}\\ 
$C_{11}^2 : C_4$ & 484 & false    &  & \{ 1, 3, 4, 5, 10 \}\\ 
$C_{11}^2 : C_4$ & 484 & false    &  & \{ 1, 3, 4, 5, 11 \}\\ 
\hline
$C_{11}^2 : C_4$ & 484 & false    & $\left[\begin{smallmatrix} 0 & 1 & 0 & 1 & 4 & 0 &  \\ 0 & 0 & 1 & 10 & 1 & 0 &  \\  \end{smallmatrix}\right]$ & \{ 1, 3, 4, 5, 6 \}\\ 
$C_{11}^2 : Q_8$ & 968 & false    &  & \{ 1, 3, 4, 5, 7 \}\\ 
$C_{11}^2 : C_{20}$ & 2420 & false    &  & \{ 1, 3, 4, 5, 8 \}\\ 
$C_{11}^2 : Q_8$ & 968 & false    &  & \{ 1, 3, 4, 5, 9 \}\\ 
$C_{11}^2 : C_4$ & 484 & false    &  & \{ 1, 3, 4, 5, 10 \}\\ 
$C_{11}^2 : C_4$ & 484 & false    &  & \{ 1, 3, 4, 5, 11 \}\\ 

\hline
\caption{Type I hemisystems of $\fish{11}$.}
\label{tab:Type1Fi11}
\end{longtable}
\end{center}
}

\subsection{Type I hemisystems of $\pentmon{11}$}

{\small
\begin{center}
\begin{longtable}{l|r|c|c|l}
Group & \multicolumn{1}{|c|}{Size}& SC &  $\ell$ & Subset ${\mathcal S}$\\ 
\hline
$C_{11}^2 : C_4$ & 484 & false  & $\left[\begin{smallmatrix} 0 & 1 & 0 & 0 & 0 & 0 &  \\ 0 & 0 & 1 & 1 & 0 & 0 &  \\  \end{smallmatrix}\right]$ & \{ 1, 3, 4, 5, 6 \}\\ 
$C_{11}^2 : Q_8$ & 968 & false  &  & \{ 1, 3, 4, 5, 7 \}\\ 
$C_{11}^2 : C_{20}$ & 2420 & false  &  & \{ 1, 3, 4, 5, 8 \}\\ 
$C_{11}^2 : Q_8$ & 968 & false  &  & \{ 1, 3, 4, 5, 9 \}\\ 
$C_{11}^2 : C_4$ & 484 & false  &  & \{ 1, 3, 4, 5, 10 \}\\ 
$C_{11}^2 : C_4$ & 484 & false  &  & \{ 1, 3, 4, 5, 11 \}\\ 
\hline
$C_{11}^2 : C_4$ & 484 & false  & $\left[\begin{smallmatrix} 0 & 1 & 0 & 0 & 0 & 0 &  \\ 0 & 0 & 1 & 2 & 0 & 0 &  \\  \end{smallmatrix}\right]$ & \{ 1, 3, 4, 5, 6 \}\\ 
$C_{11}^2 : Q_8$ & 968 & false  &  & \{ 1, 3, 4, 5, 7 \}\\ 
$C_{11}^2 : C_{20}$ & 2420 & false  &  & \{ 1, 3, 4, 5, 8 \}\\ 
$C_{11}^2 : Q_8$ & 968 & false  &  & \{ 1, 3, 4, 5, 9 \}\\ 
$C_{11}^2 : C_4$ & 484 & false  &  & \{ 1, 3, 4, 5, 10 \}\\ 
$C_{11}^2 : C_4$ & 484 & false  &  & \{ 1, 3, 4, 5, 11 \}\\ 
\hline
$C_{11}^2 : C_4$ & 484 & true  & $\left[\begin{smallmatrix} 0 & 1 & 0 & 0 & 0 & 0 &  \\ 0 & 0 & 1 & 5 & 0 & 0 &  \\  \end{smallmatrix}\right]$ & \{ 1, 3, 4, 5, 6 \}\\ 
$C_{11}^2 : Q_8$ & 968 & true  &  & \{ 1, 3, 4, 5, 7 \}\\ 
$C_{11}^2 : C_{20}$ & 2420 & true  &  & \{ 1, 3, 4, 5, 8 \}\\ 
$C_{11}^2 : Q_8$ & 968 & true  &  & \{ 1, 3, 4, 5, 9 \}\\ 
$C_{11}^2 : C_4$ & 484 & true  &  & \{ 1, 3, 4, 5, 10 \}\\ 
$C_{11}^2 : C_4$ & 484 & true  &  & \{ 1, 3, 4, 5, 11 \}\\ 
\hline
$C_{11}^2 : C_2$ & 242 & false  & $\left[\begin{smallmatrix} 0 & 1 & 0 & 0 & 0 & 0 &  \\ 0 & 0 & 1 & 10 & 0 & 0 &  \\  \end{smallmatrix}\right]$ & \{ 1, 3, 4, 5, 6 \}\\ 
$C_{11}^2 : C_2$ & 242 & false  &  & \{ 1, 3, 4, 5, 7 \}\\ 
$C_{11}^2 : C_{10}$ & 1210 & false  &  & \{ 1, 3, 4, 5, 8 \}\\ 
$C_{11}^2 : C_2$ & 242 & false  &  & \{ 1, 3, 4, 5, 9 \}\\ 
$C_{11}^2 : C_2$ & 242 & false  &  & \{ 1, 3, 4, 5, 10 \}\\ 
$C_{11}^2 : C_2$ & 242 & false  &  & \{ 1, 3, 4, 5, 11 \}\\ 
$C_{11}^2 : C_2$ & 242 & false  &  & \{ 1, 3, 4, 5, 12 \}\\ 
$C_{11}^2 : C_2$ & 242 & false  &  & \{ 1, 3, 4, 6, 10 \}\\ 
$C_{11}^2 : C_2$ & 242 & false  &  & \{ 1, 3, 4, 7, 10 \}\\ 
$C_{11}^2 : C_{10}$ & 1210 & false  &  & \{ 1, 3, 4, 7, 12 \}\\ 
\hline
$C_{11}^2 : C_2$ & 242 & false  & $\left[\begin{smallmatrix} 0 & 1 & 0 & 0 & 0 & 0 &  \\ 0 & 0 & 1 & 3 & 0 & 0 &  \\  \end{smallmatrix}\right]$ & \{ 1, 3, 4, 5, 6 \}\\ 
$C_{11}^2 : C_4$ & 484 & false  &  & \{ 1, 3, 4, 5, 7 \}\\ 
$C_{11}^2 : C_{10}$ & 1210 & false  &  & \{ 1, 3, 4, 5, 8 \}\\ 
$C_{11}^2 : C_4$ & 484 & false  &  & \{ 1, 3, 4, 5, 9 \}\\ 
$C_{11}^2 : C_2$ & 242 & false  &  & \{ 1, 3, 4, 5, 10 \}\\ 
$C_{11}^2 : C_2$ & 242 & false  &  & \{ 1, 3, 4, 5, 11 \}\\ 
\hline
$C_{11}^2 : C_2$ & 242 & false  & $\left[\begin{smallmatrix} 0 & 1 & 0 & 1 & 0 & 0 &  \\ 0 & 0 & 1 & 0 & 1 & 0 &  \\  \end{smallmatrix}\right]$ & \{ 1, 3, 4, 5, 6 \}\\ 
$C_{11}^2 : C_4$ & 484 & false  &  & \{ 1, 3, 4, 5, 7 \}\\ 
$C_{11}^2 : C_{10}$ & 1210 & false  &  & \{ 1, 3, 4, 5, 8 \}\\ 
$C_{11}^2 : C_4$ & 484 & false  &  & \{ 1, 3, 4, 5, 9 \}\\ 
$C_{11}^2 : C_2$ & 242 & false  &  & \{ 1, 3, 4, 5, 10 \}\\ 
$C_{11}^2 : C_2$ & 242 & false  &  & \{ 1, 3, 4, 5, 11 \}\\ 
\hline
$C_{11}^2 : C_2$ & 242 & false  & $\left[\begin{smallmatrix} 0 & 1 & 0 & 1 & 1 & 0 &  \\ 0 & 0 & 1 & 1 & 1 & 0 &  \\  \end{smallmatrix}\right]$ & \{ 1, 3, 4, 5, 6 \}\\ 
$C_{11}^2 : C_2$ & 242 & true  &  & \{ 1, 3, 4, 5, 7 \}\\ 
$C_{11}^2 : C_{10}$ & 1210 & false  &  & \{ 1, 3, 4, 5, 8 \}\\ 
$C_{11}^2 : C_2$ & 242 & true  &  & \{ 1, 3, 4, 5, 9 \}\\ 
$C_{11}^2 : C_2$ & 242 & false  &  & \{ 1, 3, 4, 5, 10 \}\\ 
$C_{11}^2 : C_2$ & 242 & false  &  & \{ 1, 3, 4, 5, 11 \}\\ 
\hline
$C_{11}^2 : Q_8$ & 968 & false  & $\left[\begin{smallmatrix} 0 & 1 & 0 & 1 & 1 & 0 &  \\ 0 & 0 & 1 & 4 & 1 & 0 &  \\  \end{smallmatrix}\right]$ & \{ 1, 3, 4, 5, 6 \}\\ 
$C_{11}^2 : C_2 \times Q_8$ & 1936 & false  &  & \{ 1, 3, 4, 5, 7 \}\\ 
$C_{11}^2 : C_5 \times Q_8$ & 4840 & false  &  & \{ 1, 3, 4, 5, 8 \}\\ 
$C_{11}^2 : C_2 \times Q_8$ & 1936 & false  &  & \{ 1, 3, 4, 5, 9 \}\\ 
$C_{11}^2 : Q_8$ & 968 & false  &  & \{ 1, 3, 4, 5, 10 \}\\ 
$C_{11}^2 : Q_8$ & 968 & false  &  & \{ 1, 3, 4, 5, 11 \}\\ 
\hline
$C_{11}^2 : C_6$ & 726 & false  & $\left[\begin{smallmatrix} 0 & 1 & 0 & 1 & 8 & 0 &  \\ 0 & 0 & 1 & 8 & 1 & 0 &  \\  \end{smallmatrix}\right]$ & \{ 1, 3, 4, 5, 6 \}\\ 
$C_{11}^2 : C_3 : C_4$ & 1452 & false  &  & \{ 1, 3, 4, 5, 7 \}\\ 
$C_{11}^2 : C_{30}$ & 3630 & false  &  & \{ 1, 3, 4, 5, 8 \}\\ 
$C_{11}^2 : C_3 : C_4$ & 1452 & false  &  & \{ 1, 3, 4, 5, 9 \}\\ 
$C_{11}^2 : C_6$ & 726 & false  &  & \{ 1, 3, 4, 5, 10 \}\\ 
$C_{11}^2 : C_6$ & 726 & false  &  & \{ 1, 3, 4, 5, 11 \}\\ 
\hline
$C_{11}^2 : C_2$ & 242 & false  & $\left[\begin{smallmatrix} 0 & 1 & 0 & 1 & 8 & 0 &  \\ 0 & 0 & 1 & 7 & 1 & 0 &  \\  \end{smallmatrix}\right]$ & \{ 1, 3, 4, 5, 6 \}\\ 
$C_{11}^2 : C_4$ & 484 & false  &  & \{ 1, 3, 4, 5, 7 \}\\ 
$C_{11}^2 : C_{10}$ & 1210 & false  &  & \{ 1, 3, 4, 5, 8 \}\\ 
$C_{11}^2 : C_4$ & 484 & false  &  & \{ 1, 3, 4, 5, 9 \}\\ 
$C_{11}^2 : C_2$ & 242 & false  &  & \{ 1, 3, 4, 5, 10 \}\\ 
$C_{11}^2 : C_2$ & 242 & false  &  & \{ 1, 3, 4, 5, 11 \}\\ 
\hline
$C_{11}^2 : C_2$ & 242 & false  & $\left[\begin{smallmatrix} 0 & 1 & 0 & 1 & 5 & 0 &  \\ 0 & 0 & 1 & 7 & 1 & 0 &  \\  \end{smallmatrix}\right]$ & \{ 1, 3, 4, 5, 6 \}\\ 
$C_{11}^2 : C_4$ & 484 & false  &  & \{ 1, 3, 4, 5, 7 \}\\ 
$C_{11}^2 : C_{10}$ & 1210 & false  &  & \{ 1, 3, 4, 5, 8 \}\\ 
$C_{11}^2 : C_4$ & 484 & false  &  & \{ 1, 3, 4, 5, 9 \}\\ 
$C_{11}^2 : C_2$ & 242 & false  &  & \{ 1, 3, 4, 5, 10 \}\\ 
$C_{11}^2 : C_2$ & 242 & false  &  & \{ 1, 3, 4, 5, 11 \}\\ 
\hline
$C_{11}^2 : C_3 : C_4$ & 1452 & false  & $\left[\begin{smallmatrix} 0 & 1 & 0 & 1 & 10 & 0 &  \\ 0 & 0 & 1 & 9 & 1 & 0 &  \\  \end{smallmatrix}\right]$ & \{ 1, 3, 4, 5, 6 \}\\ 
$C_{11}^2 : C_3 : Q_8$ & 2904 & false  &  & \{ 1, 3, 4, 5, 7 \}\\ 
$C_{11}^2 : C_5 \times (C_3 : C_4)$ & 7260 & false  &  & \{ 1, 3, 4, 5, 8 \}\\ 
$C_{11}^2 : C_3 : Q_8$ & 2904 & false  &  & \{ 1, 3, 4, 5, 9 \}\\ 
$C_{11}^2 : C_3 : C_4$ & 1452 & false  &  & \{ 1, 3, 4, 5, 10 \}\\ 
$C_{11}^2 : C_3 : C_4$ & 1452 & false  &  & \{ 1, 3, 4, 5, 11 \}\\ 
\hline
$C_{11}^2 : C_3 : C_4$ & 1452 & false  & $\left[\begin{smallmatrix} 0 & 1 & 0 & 2 & 2 & 0 &  \\ 0 & 0 & 1 & 0 & 2 & 0 &  \\  \end{smallmatrix}\right]$ & \{ 1, 3, 4, 5, 6 \}\\ 
$C_{11}^2 : C_3 : Q_8$ & 2904 & false  &  & \{ 1, 3, 4, 5, 7 \}\\ 
$C_{11}^2 : C_5 \times (C_3 : C_4)$ & 7260 & false  &  & \{ 1, 3, 4, 5, 8 \}\\ 
$C_{11}^2 : C_3 : Q_8$ & 2904 & false  &  & \{ 1, 3, 4, 5, 9 \}\\ 
$C_{11}^2 : C_3 : C_4$ & 1452 & false  &  & \{ 1, 3, 4, 5, 10 \}\\ 
$C_{11}^2 : C_3 : C_4$ & 1452 & false  &  & \{ 1, 3, 4, 5, 11 \}\\ 

\hline
\caption{Type I hemisystems of $\pentmon{11}$}
\label{tab:mon11}
\end{longtable}
\end{center}
}

\end{document}